\documentclass[review,onefignum,onetabnum]{siamart171218}

\usepackage{lipsum}
\usepackage{amsfonts}
\usepackage{graphicx}
\usepackage{epstopdf}
\usepackage{algorithmic}
\usepackage{booktabs}
\usepackage{multirow}
\usepackage{adjustbox} 
\ifpdf
  \DeclareGraphicsExtensions{.eps,.pdf,.png,.jpg}
\else
  \DeclareGraphicsExtensions{.eps}
\fi

\newsiamremark{assumption}{Assumption}
\newsiamremark{remark}{Remark}
\newsiamremark{hypothesis}{Hypothesis}
\crefname{hypothesis}{Hypothesis}{Hypotheses}
\newsiamthm{claim}{Claim}

\headers{A parameterized DR  splitting  for nonconvex  optimization}{Fengmiao Bian and Xiaoqun Zhang}

\title{A parameterized Douglas-Rachford  splitting  algorithm for nonconvex  optimization
}

\author{Fengmiao Bian\thanks{School of Mathematical Sciences, Shanghai Jiao Tong University, Shanghai 200240, CHINA
  (\email{bianfm17@sjtu.edu.cn}).}
\and Xiaoqun Zhang\thanks{School of Mathematical Sciences and Institute of Natural Sciences, Shanghai Jiao Tong University, Shanghai 200240, CHINA
  (\email{xqzhang@sjtu.edu.cn}).}}

\usepackage{amsopn}

\ifpdf
\hypersetup{
  pdftitle={A parameterized Douglas-Rachford  splitting  algorithm for nonconvex  optimization},
  pdfauthor={Fengmiao Bian and Xiaoqun Zhang} 
}
\fi

\begin{document}

\maketitle

\begin{abstract}
 In this paper, we  study a parameterized Douglas-Rachford splitting method for  a class of nonconvex optimization problem. A new merit function is constructed to establish the convergence of the whole sequence generated by the parameterized Douglas-Rachford splitting method.  We then apply the parameterized Douglas-Rachford splitting method to three important classes of nonconvex optimization problems arising in data science: sparsity constrained least squares problem, feasibility problem and low rank matrix completion. Numerical results validate the effectiveness of the parameterized  Douglas-Rachford splitting method compared with some other classical methods.
 \end{abstract}

\begin{keywords}
 Parameterized  Douglas-Rachford splitting method; nonconvex optimization problems; global convergence; sparsity constrained least squares problem;  low rank matrix completion;  feasibility problem
\end{keywords}

\begin{AMS}
90C26, 90C30, 90C90,15A83,65K05
\end{AMS}

\section{Introduction}
The Douglas-Rachford (DR) splitting is a common splitting method, which was originally introduced in \cite{DR} for finding numerical solutions of heat differential equations. Later, Lions and Mercier \cite{LMe} used this method  to minimize the sum of two closed convex functions
\begin{equation}\label{model}
\min_{u} f(u) + g(u),
\end{equation}
 by solving the optimality condition $0 \in \partial f(u) + \partial g(u).$ In the case when $f$ and $g$ are both convex, the DR splitting method can be written by the following interation:
\begin{equation}\label{drsch}
x^{t+1} = \frac{1}{2}x^{t} + \frac{1}{2}(2~\textmd{prox}_{\gamma g} - I) \circ (2~\textmd{prox}_{\gamma f} - I) (x^t),
\end{equation}
where $I$ is the identity mapping, $\gamma > 0$ and
\begin{equation}\label{def-pro}
\textmd{prox}_{\gamma f} (x) := \arg\min_{u} f(u) + \frac{1}{2\gamma} \| u - x \|^2.
\end{equation}
From \eqref{drsch} we can see that DR splitting method converts solving problem \eqref{model} into solving two proximal operators that are easy to compute. Therefore, the Douglas–Rachford (DR) splitting method is a very powerful algorithm to solving problems with competing structures, such as finding a point in the intersection of two closed convex sets (feasibility problem), which can be viewed as a minimization problem that minimizes the sum of the indicator functions of the two sets. Eckstein and Bertsekas \cite{EB} studied further the DR splitting algorithm in the convex setting and  revealed  its relationship with the proximal point algorithm.  The DR splitting method has also been applied to various convex optimization problems arising from signal processing and other applications; see, for example, \cite{CP, GRY, HY, PSB}. 

Recently, in \cite{AC}, Aragon Artacho and Campoy studied a slight modification of the Douglas–Rachford method to the best approximation problem of finding the closest point in $ A \cap B$ to $q$, i.e.,
\begin{equation}\label{clsp}
\min_{x \in A \cap B} \| x - q \|,
\end{equation}
where $A$ and $B$ are two nonempty closed convex sets of a Hilbert space $\mathcal{H}$ and $q \in \mathcal{H}$ is any given point. More specifically, for any pair of parameters $\eta, \beta \in (0, 1)$, they introduced an averaged alternating modified reflections operator (AAMR operator) $T_{A, B, \eta, \beta}: \mathcal{H} \to \mathcal{H}$
\begin{equation}\label{T-oper}
T_{A, B, \eta, \beta} := (1 - \eta) I + \eta (2 \beta P_B - I)(2 \beta P_A - I),
\end{equation}
and defined a new projection method termed averaged alternating modified reflections (AAMR) method, which generated the sequence iteratively by
\begin{equation}\label{AAMR}
x^{t+1} := T_{A, B, \eta, \beta} (x^t), ~~~~t = 0, 1, 2, \dots 
\end{equation}
where $I$ denotes the identity mapping and $P_A$, $P_B$ denote the projectors onto A and B, respectively. In \cite{WW}, Wang and Wang proposed another similar modification to the DR method called parameterized Douglas–Rachford (PDR) algorithm which is iteratively defined in the convex setting by 
\begin{equation}\label{b-2}
x^{t+1} = \Big[ \left(1 - \frac{1}{\alpha} \right)I + \frac{1}{\alpha} \left( \alpha~\textmd{prox}_{\gamma g} - I \right) \left( \alpha~\textmd{prox}_{\gamma f} - I \right) \Big] (x^t),
\end{equation}
where $\alpha \in (1, 2].$ Obviously, in the case when $f$ and $g$ are the indicator functions of the two convex sets, $\eta = \frac{1}{\alpha}$ and $\alpha = 2 \beta$, the PDR method is exactly the AAMR method. We also notice that when $\eta = \frac{1}{2}$, $\beta = 1$ in the AAMA method and $\alpha = \frac{1}{2}$ in the PDR method, these two methods are the classical DR method. When $\eta = 1, \beta = 1$, the  AAMR method is the Peaceman–Rachford (PR) splitting method (see \cite{PR}) to solving feasibility problem in the convex setting. As shown in papers \cite{WW,AC}, surprisingly, though, the slight modification
$(\alpha\textmd{prox}_{\gamma f} - I)$ and $(\alpha\textmd{prox}_{\gamma g} - I)$ in the reflector operators $(2\textmd{prox}_{\gamma f} - I)$ and $(2\textmd{prox}_{\gamma g} - I)$ completely changes the dynamics of the sequence generated by the scheme.


%

On the other hand, the behavior of the DR splitting algorithm in the nonconvex cases has attracted a lot of attention. A very important reason is that DR method has been successfully applied to many important practical nonconvex problems \cite{ABT1,HL,HLN,P,LP}. For example, Hesse and Luke showed in \cite{HL} that the DR  splitting method exhibits local linear convergence for an affine set and a super-regular set,   see also \cite{P} for  the similar results  of DR splitting method for two super-regular sets.  Li and Pong \cite{LP} showed that DR splitting method can be applied to the nonconvex problem \eqref{model}  under some assumptions for $f$ and $g$,  and then used this nonconvex DR splitting method to find a point in the intersection of a closed convex set $C$ and a general closed set $D$. Very recently, Themelis and Patrinos \cite{TP} gave a unified method to prove the convergence for ADMM and DR splitting applied to nonconvex problems under less restrictive assumptions than \cite{LP}.  

Naturally, in this paper we want to understand the convegence of the PDR splitting method as stated above in the nonconvex case and explore the performance of  PDR splitting method for many important nonconvex optimization problems arising in machine learning. Firstly, we show that, if the step-size parameter is smaller than a computable threshold and the sequence generated by the PDR splitting method has a cluster point, then it gives a stationary point of problem \eqref{model} where $f$ and $g$ are both possible nonconvex. We also present some sufficient conditions to guarantee the boundedness of the sequence generated by the PDR splitting, and so the existence of cluster point. Furthermore, we show the convergence of the whole sequence generated by the PDR splitting under the additional assumption that $f$ and $g$ are semi-algebraic. Finally, the PDR splitting method is appllied to solving the sparsity least square problem, feasibility problem and low rank matrix completion problem. Our preliminary numerical results show that the
PDR splitting method outperforms the DR and PR methods.


The rest of this paper is structured as follows. In \cref{sec:notation}, we give some notations  and preliminary materials.  In \cref{sec:convergence}, we introduce the  parameterized Douglas-Rachford splitting algorithm and establish the convergence of this algorithm for nonconvex optimization problems where the objective function is the sum of a smooth function $f$ with Lipschitz continuous  gradient and a closed function $g$. In \cref{sec:examples}, as stated before, we demonstrate how the PDR splitting method can be applied well for solving several important classes of nonconvex optimization problems arising in machine learning.  In \cref{sec:conclude}, we give some concluding remarks.\\

\section{Notation and preliminaries}
\label{sec:notation}

In this section, we introduce our notations and state some basic concepts. We refer the reader to the textbooks \cite{RW, BC} for these basic knowledge.\\

 Let $\mathbb{R}^{n}$ denote the $n$-dimensional Euclidean space, $ \left \langle \cdot, \cdot \right \rangle$ denote the inner product and the induced norm by $\| \cdot \| = \sqrt{\left \langle \cdot, \cdot \right \rangle}$. For an extended-real-valued function $f : \mathbb{R}^{n} \to (-\infty, \infty]$, we say that $f$ is proper if it is never $-\infty$ and its domain, dom $f := \{x \in \mathbb{R}^{n} : f(x) < +\infty\}$, is nonempty. The function is called closed if it is proper and lower semicontinuous. 
 
\begin{definition}\label{DEF-201} $($limiting  subdifferential$)$. 
Let $f$  be a  proper function.  The limiting {\it subdifferential} of $f$ at $x \in $ dom $f$ is defined by

\begin{equation}\label{subdiff}
\aligned
\partial f(x) :=  \Bigg\{  &  v \in \mathbb{R}^{n} : \exists x^{t} \to x, f(x^{t}) \to f(x), v^{t} \to v ~~ \textmd{with} \\
& \liminf_{z \to x^t} \frac{f(z) - f(x^t) -  \left \langle v^t, z - x^t \right \rangle}{\| z - x^t \|} \geq 0 ~ \textmd{for} ~ \textmd{each} ~ t       \Bigg\}. 
\endaligned
\end{equation}  
\end{definition} 

If $f$ is differentiable at $x$, we have $\partial f(x) = \{ \nabla f(x) \}$. If $f$ is convex, we have 
\begin{equation}\label{condiff}
\partial f(x) = \Bigg\{ v \in \mathbb{R}^{n} : f(z) \geq f(x) + \left \langle v, z-x \right \rangle~\textmd{for}~\textmd{any}~ z \in \mathbb{R}^{n} \Bigg\},
\end{equation}
which is the classic definition of subdifferential in convex analysis. Moreover, we also have the following robustness property:
\begin{equation}\label{subdiffproperty}
\Bigg\{  v \in \mathbb{R}^{n} :  \exists x^{t} \to x, f(x^{t}) \to f(x), v^{t} \to v, v^{t} \in \partial f(x^{t})\Bigg\}\subseteq \partial f(x).
\end{equation}
A point $x^{*}$ is a stationary point of a function $f$ if $0 \in \partial f(x^{*})$. $x^{*}$ is a critical point of $f$ if $f$ is differentiable at $x^{*}$ and $\nabla f(x^{*}) = 0$. A function is called to be coercive if $\liminf_{\| x \| \to \infty} f(x) = \infty$. We say that a function $f$ is a strongly convex function with modulus $\sigma > 0$ if $f - \frac{\sigma}{2} \| \cdot \|^{2}$ is a convex function.\\

For any $\gamma >0$, the proximal mapping of $f$ is defined by 
$$
P_{\gamma f} (x): x \rightarrow {\arg\min}_{y\in\mathbb{R}^n} \left\{ f(y) + \frac{1}{2\gamma} \| y - x\|^2 \right\}, 
$$
assuming that the $\arg \min$ exists,  where $\rightarrow$  means a possibly set-valued mapping.  

\begin{definition}\label{Indi-01} $($indicator function$)$. 
For a closed set $S \subseteq \mathbb{R}^{n}$, its indicator function $\mathcal{I}_{S}$ is defined by
$$
\mathcal{I}_S(x)= 
\begin{cases}
 0, ~~~~ & \textmd{if} ~ x\in S,\\
+\infty, ~~~~ & \textmd{if} ~ x\notin S. 
\end{cases}
$$ 
\end{definition} 

\begin{definition}\label{Sem-201} $($real semialgebraic set$)$. 
A semi-algebraic set $S \subseteq \mathbb{R}^{n}$ is a finite union of sets of the form
\begin{equation}\label{semialg}
\Bigg\{ x \in \mathbb{R}^{n} : h_{1} (x) = \cdots h_{k} (x) = 0, ~g_{1} (x) < 0, \dots , g_{l} (x) < 0 \Bigg\},
\end{equation}
where $g_{1}, \dots, g_{l}$ and $h_{1}, \dots, h_{k}$ are real polynomials. 
\end{definition} 
\begin{definition}\label{Sem-202} $($real semialgebraic function$)$. 
A function $F : \mathbb{R}^{n} \to \mathbb{R}$ is semi-algebraic if the set $\big\{ (x,~F(x)) \in \mathbb{R}^{n+1} : x \in \mathbb{R}^{n} \big\}$ is semi-algebraic.
\end{definition} 

Remark that the semi-algebraic sets and semi-algebraic functions can be easily identified and contain a large number of possibly nonconvex functions arising in applications, such as see \cite{ABRS,ABS,BDL}. We will also use the following Kurdyka-\L ojasiewicz (KL) property which holds in particular for semi-algebraic functions.\\

\begin{definition}\label{KL}$($KL property and KL function$)$.
The function $F: \mathbb{R}^{n} \to \mathbb{R} \cup \{\infty\}$ has the Kurdyka-\L ojasiewicz property at $x^{*} \in$ dom $\partial F$ if there exist $\eta \in (0, \infty]$, a
neighborhood $U$ of $x^{*}$, and a continuous concave function $\varphi : [0, \eta) \to \mathbb{R}_{+}$ such that:

\begin{itemize}
\item [(i)]  $\varphi(0) = 0,~\varphi \in C^{1}((0, \eta))$, and  $\varphi^{'}(s) > 0$ for all $s \in (0, \eta)$;
\item[(ii)] for all $x \in U \cap [F(x^{*}) < F < F(x^{*})+ \eta]$ the Kurdyka-\L ojasiewicz inequality holds, i.e.,
$$ 
\varphi^{'}(F(x) - F(x^{*}))dist(0, \partial F(x)) \geq 1.
$$
\end{itemize}
If the function $F$ satisfies the Kurdyka-\L ojasiewicz property  at each point of dom $\partial F$, it is called a KL function.
\end{definition}
\begin{remark}
It follows from \cite{ABRS} that a proper closed semi-algebraic function always satisfies the KL property.
\end{remark}

\section{ Parameterized Douglas-Rachford splitting algorithm and its convergence}
\label{sec:convergence}
In this section, we consider the following optimization problem:
\begin{equation}\label{3-model}
\min_{u}    f(u) + g(u),
\end{equation}
where $f$ and $g$ are proper closed possibly nonconvex functions.  When parameterized  Douglas-Rachford splitting  is applied to problem \eqref{3-model}, we can give the algorithm in \cref{alg:GDR}. 

\begin{algorithm}
\caption{Parameterized Douglas-Rachford Splitting Algorithm}
\label{alg:GDR}
\begin{algorithmic}
\STATE{{\bf{Step 0.}}  Choose a step-size $\gamma >0$,~$\alpha\in(\frac{3}{2},2]$ and an initial point $x^{0}.$}

\STATE{{\bf{Step 1.}}  Set
\begin{equation}\label{alg}
\begin{cases}
u^{t+1} \in  \arg\min_{u}  \Bigg\{  f(u) + \frac{1}{2\gamma}\|u-x^{t}\|^{2} \Bigg\},
\\
v^{t+1} \in  \arg\min_{v} \Bigg\{ g(v) + \frac{1}{2\gamma}\|v - \alpha u^{t+1} + x^{t} \|^{2} \Bigg\},\\
\\
x^{t+1} = x^{t} + (v^{t+1} - u^{t+1}).
\end{cases}
\end{equation}
}
\STATE{{\bf{Step 2.}}  If a termination criterion is not met, go to Step 1.}

\end{algorithmic}
\end{algorithm}

By using the optimal conditions and the subdifferential calculus rule \cite{RW} for $u$ and $v$-updates  in \eqref{alg} we have
\begin{equation}\label{opt01}
0= \nabla f(u^{t+1}) + \frac{1}{\gamma} (u^{t+1} - x^t), 
\end{equation}
\begin{equation}\label{opt02}
0 \in \partial g (v^{t+1}) + \frac{1}{\gamma} (v^{t+1} + x^t - \alpha u^{t+1}). 
\end{equation} 

We notice that the Lipschitz differentiability of $f$ is very important to the recent convergence analysis of the DR splitting method in the nonconvex settings, for instance in \cite{LP,TP}. In addition, the  proximal mapping of $g$  is well-defined and easy to compute. Hence it's reasonable to make the following assumptions about $f$ and $g$.

\begin{assumption}\label{ass1} Functions $f, g$ satisfy
\begin{itemize}
\item [(a1)] The function $f$ has a Lipschitz continuous gradient, i.e, there exists a constant $L  > 0$ such that
\begin{equation}\label{flip}
\| \nabla f(y_1) - \nabla f(y_2) \| \leq L \| y_1 - y_2 \|, ~~~\forall  y_1, y_2 \in \mathbb{R}^{n};
\end{equation}
\item[(a2)] $g$ is a proper closed function with a nonempty mapping $P_{\gamma g}(x)$ for any $x$ and for $\gamma > 0$.
\end{itemize}
\end{assumption}

\begin{remark}We point out that
\begin{itemize}
\item[(i)] In \cite{LP,TP}, under the condition that $f$ has a Lipschitz continuous gradient, the authors established the convergence of DR splitting. In \cite{LLP}, under the conditions that $f$ has a Lipschitz continuous gradient and is strongly convex, the authors proved the convergence of PR splitting;
\item[(ii)]Under the assumption $(a1)$, we can always find $l \in \mathbb{R}$ such that $f + \frac{l}{2}\| \cdot \|^{2}$ is convex, in particular, we can take $l = L$.
\end{itemize}
\end{remark}

By the optimal condition \eqref{opt01} and the Lipschitz continuity of $\nabla f$, we can easily get the following lemma. 

\begin{lemma}\label{opt-con39} 
Suppose function $f$ has a Lipschitz continuous gradient whose Lipschitz continuity constant is bounded by $L > 0$. Then the sequence  $\{(u^t, v^t, x^t)\}$ generated by \eqref{alg} satisfies 
\begin{equation}\label{ineq4}
\| x^{t} - x^{t-1} \| \leq ( 1 + \gamma L ) \| u^{t+1} - u^{t} \|.
\end{equation}
\end{lemma} 

Inspired by \cite{LP}, To analyze the convergent behavior of PDR splitting method, we construct the following merit function:
\begin{equation}\label{meritfun}
\mathcal{M}_{\gamma} (u, v, x) := f(u) + g(v) - \frac{1}{2\gamma}\| u - v \|^{2} + \frac{1}{\gamma} \left\langle x - (\alpha - 1)u, v - u \right\rangle + \frac{2 - \alpha}{2\gamma}\| u \|^{2}.
\end{equation} 
Moreover, it is not hard to see that the merit function $\mathcal{M}_{\gamma}$ can be alternatively written as
\begin{subequations}
\begin{align}
  \mathcal{M}_{\gamma} (u, v, x) 
  &=   f(u) + g(v) + \frac{1}{2\gamma} \| \alpha u - v - x \|^{2} - \frac{1}{2\gamma} \| x - (\alpha - 1)u \|^{2}  \label{meritfun1}  \\
  &~~~~ - \frac{1}{\gamma} \| u - v \|^{2} + \frac{2 - \alpha}{2\gamma} \| u \|^{2}
  \nonumber \\
  &= f(u) + g(v) + \frac{1}{2\gamma} \Big( \|x - u\|^2 - \|x - v\|^2 \Big)  + \frac{1}{\gamma}  \left\langle (2 - \alpha) u,   v - u \right\rangle  \label{meritfun58} \\
  &~~~~ + \frac{2 - \alpha}{2\gamma} \| u \|^{2} \nonumber ,
\end{align}
\end{subequations}
where the Equation  \eqref{meritfun1}  follows from the elementary relation $\left\langle a,  b\right\rangle = \frac{1}{2} ( \| a + b \|^{2} - \| a \|^{2} - \| b \|^{2})$ applied with $a = x - (\alpha - 1)u$ and $b = v - u$ in \eqref{meritfun}, and the Equation \eqref{meritfun58} follows from the relation $\left\langle a, b \right\rangle = \frac{1}{2}(\| a \|^{2} + \| b \|^{2} - \|a - b\|^{2})$ in \eqref{meritfun} with $a = x - u$ and $b = v - u$. These equivalent relations will be made use of in our convergence analysis.

Next, we state and prove a convergence theorem for the parameterized Douglas-Rachford splitting algorithm \eqref{alg}. 
\begin{theorem}\label{the1} $($Global subsequential convergence$)$
Assume that the functions $f$ and $g$ satisfy \cref{ass1}. Let $\frac{3}{2} < \alpha \leq 2$ and the parameter $\gamma > 0$ is chosen such that
\begin{equation}\label{gammcond}
\frac{4 - \alpha}{2}( 1 + \gamma L)^{2} + \frac{9 - 2\alpha}{2}\gamma l - \frac{1 + \alpha}{2} < 0.
\end{equation}
Then the sequence $\{ \mathcal{M}_{\gamma} ( u^{t}, v^{t}, x^{t})\}$ is nonincreasing. Moreover, if a cluster point of the sequence $\{ ( u^{t}, v^{t}, x^{t}) \}$ exists, then
\begin{equation}\label{lim}
lim_{t \to \infty} \| x^{t+1} - x^{t} \| = lim_{t \to \infty} \| v^{t+1} - u^{t+1} \| = 0.
\end{equation}
Furthermore, for any cluster point $(u^{*}, v^{*}, x^{*})$, we have $u^{*} = v^{*}$, and 
\begin{equation}\label{lastopti}
0 \in \nabla f(v^{*}) + \partial g(v^{*}) + \frac{1}{\gamma}(2 - \alpha) v^{*}.
\end{equation}
\end{theorem}
\begin{remark}\label{remark01}Notice that
\begin{itemize}
\item[1.] The $\gamma$ that satisfies \eqref{gammcond} always exists when $\frac{3}{2} < \alpha \leq 2$. Indeed, if $\frac{3}{2} < \alpha \leq 2$, then 
\begin{equation}
\lim_{\gamma \to 0}  \left[ \frac{4 - \alpha}{2}( 1 + \gamma L)^{2} + \frac{9 - 2\alpha}{2}\gamma l - \frac{1 + \alpha}{2} \right] = \frac{3 - 2\alpha}{2}  < 0.
\end{equation}
Therefore, given $l \in \mathbb{R}$, $L > 0$ and $\frac{3}{2} < \alpha \leq 2$, we can always find some $\gamma_0$ such that the condition \eqref{gammcond} holds for $\gamma \in (0, \gamma_0)$;
\item[2.]  From \eqref{lastopti} we know that when \cref{alg:GDR} is applied to problem \eqref{3-model}, any cluster point of the sequence generated by \cref{alg:GDR} is a critical point of problem  \eqref{3-model} with an additional regularization term $\frac{2 - \alpha}{2\gamma}\| u \|^2$. Therefore, if we substitute $\tilde{g} = g - \frac{2 - \alpha}{2\gamma} \| u \|^2$ for $g$ in \cref{alg:GDR}, the corresponding clustering point is the critical point of problem \eqref{3-model}.
\end{itemize} 
\end{remark} 

\begin{proof} We first verify the nonincreasing property of  $\mathcal{M}_{\gamma} ( u^{t}, v^{t}, x^{t} )$. Firstly, using the equation \eqref{meritfun} and the definition of $x$-update, we can get
\begin{equation}\label{x}
\aligned
&\mathcal{M}_{\gamma}(u^{t+1}, v^{t+1}, x^{t+1}) - \mathcal{M}_{\gamma}(u^{t+1}, v^{t+1}, x^{t})\\
&= \frac{1}{\gamma} \left\langle x^{t+1} - x^{t}, v^{t+1} - u^{t+1} \right\rangle\\
&= \frac{1}{\gamma}\| x^{t+1} - x^{t} \|^{2}.
\endaligned
\end{equation}
Secondly, employing  \eqref{meritfun1} and the fact that $v^{t+1}$ is a minimizer, we have
\begin{equation}\label{v}
\aligned
&\mathcal{M}_{\gamma}(u^{t+1}, v^{t+1}, x^{t}) - \mathcal{M}_{\gamma}(u^{t+1}, v^{t}, x^{t})\\
&= g(v^{t+1}) + \frac{1}{2\gamma} \| \alpha u^{t+1} - v^{t+1} - x^{t} \|^{2} - \frac{1}{\gamma} \| u^{t+1} - v^{t+1} \|^{2}\\
& ~~~  - g(v^{t}) - \frac{1}{2\gamma} \| \alpha u^{t+1} - v^{t} - x^{t} \|^{2} + \frac{1}{\gamma} \| u^{t+1} - v^{t} \|^{2}\\
& \leq \frac{1}{\gamma}( \| u^{t+1} - v^{t} \|^{2} - \| u^{t+1} - v^{t+1} \|^{2} )\\
&= \frac{1}{\gamma} ( \| u^{t+1} - v^{t} \| ^{2} - \| x^{t+1} - x^{t} \|^{2} ),
\endaligned
\end{equation}
where the definition of $x^{t+1}$ is used in the last equality.  
Note that from equation \eqref{opt01}, we have
\begin{equation}\label{fopti}
\nabla \left( f + \frac{l}{2} \| \cdot \|^{2} \right) \left(  u^{t+1} \right)  = \frac{1}{\gamma} \left(  x^{t} - u^{t+1}  \right)  + l u^{t+1}. 
\end{equation}
Recall that $f + \frac{l}{2} \| \cdot \|^{2}$ is convex function,  by the monotonicity of the gradient of a convex function, we  obtain that for any $t \geq 1$, 
\begin{equation}\label{ineq1}
\left\langle \left( \frac{1}{\gamma} \left( x^{t} - u^{t+1} \right) + l u^{t+1} \right) - \left( \frac{1}{\gamma} \left( x^{t-1} - u^{t} \right) + l u^{t} \right),  u^{t+1} - u^{t} \right\rangle \geq 0, 
\end{equation}
which implies that 
\begin{equation} 
\left\langle u^{t+1} - u^{t}, x^{t} - x^{t-1} \right\rangle \geq ( 1 - \gamma l ) \| u^{t+1} - u^{t} \|^{2}. 
\end{equation}
Hence, we see further that
\begin{equation}\label{ineq2}
\aligned
&\| u^{t+1} - v^{t} \|^{2} = \| u^{t+1} - u^{t} + u^{t} - v^{t} \|^{2} = \| u^{t+1} - u^{t} - \left( x^{t} - x^{t-1}  \right) \|^{2}  \\
& \leq \| u^{t+1} - u^{t} \|^{2} - 2 ( 1 - \gamma l )\| u^{t+1} - u^{t} \|^{2} + \| x^{t} - x^{t-1} \|^{2}\\
& = ( -1 + 2\gamma l ) \| u^{t+1} - u^{t} \|^{2} + \| x^{t} - x^{t-1} \|^{2}. 
\endaligned
\end{equation}
Substituting \eqref{ineq2} to \eqref{v}, we obtain 
\begin{equation}\label{v2}
\aligned
&\mathcal{M}_{\gamma} ( u^{t+1}, v^{t+1}, x^{t} ) - \mathcal{M}_{\gamma} ( u^{t+1}, v^{t}, x^{t} )\\
& \leq -\frac{1}{\gamma} \| x^{t+1} - x^{t} \|^{2} + \frac{1}{\gamma} \Big[  ( -1 + 2\gamma l ) \| u^{t+1} - u^{t} \|^{2} + \| x^{t} - x^{t-1} \|^{2} \Big]. \\
\endaligned
\end{equation}
Finally, from  \eqref{meritfun58} we get that
\begin{equation}\label{u}
\aligned
&\mathcal{M}_{\gamma}( u^{t+1}, v^{t}, x^{t} ) - \mathcal{M}_{\gamma}( u^{t}, v^{t}, x^{t} )\\
&= f( u^{t+1} ) + \frac{1}{2\gamma}\| x^{t} - u^{t+1} \|^{2} + \frac{1}{\gamma}\left\langle ( 2 - \alpha ) u^{t+1}, v^{t} - u^{t+1} \right\rangle + \frac{2 - \alpha}{2\gamma}\| u^{t+1} \|^{2}\\
& ~~~ - f( u^{t} ) - \frac{1}{2\gamma}\| x^{t} - u^{t} \|^{2} - \frac{1}{\gamma}\left\langle ( 2 - \alpha ) u^{t}, v^{t} - u^{t} \right\rangle - \frac{2 - \alpha}{2\gamma}\| u^{t} \|^{2}.
\endaligned
\end{equation}
By the conditions \eqref{gammcond} and $0< \alpha \leq 2$ we have $l < \frac{1 + \alpha}{8 -2\alpha +1} \frac{1}{\gamma} < \frac{1}{\gamma}$.  Therefore,  $f + \frac{1}{2\gamma} \| x^{t} - \cdot \|^{2}$ is a strongly convex function with modulus $\frac{1}{\gamma} - l $. This together with the definition of $u^{t+1}$ gives  
\begin{equation}\label{f strconvex}
\aligned
&f( u^{t+1} ) + \frac{1}{2\gamma} \| x^{t} - u^{t+1} \|^{2} - f( u^{t} ) - \frac{1}{2\gamma} \| x^{t} - u^{t} \|^{2}\\
& \leq -\frac{1}{2} \left(  \frac{1}{\gamma} - l  \right)  \| u^{t+1} - u^{t} \|^{2}.
\endaligned
\end{equation}
Hence, we have 
\begin{equation}\label{u1}
\aligned
&\mathcal{M}_{\gamma}( u^{t+1}, v^{t}, x^{t} ) - \mathcal{M}_{\gamma}( u^{t}, v^{t}, x^{t} )\\
& \leq -\frac{1}{2} \left(  \frac{1}{\gamma} - l  \right)  \| u^{t+1} - u^{t} \|^{2} + \frac{2 - \alpha}{\gamma} \left\langle u^{t+1}, v^{t} - u^{t+1} \right\rangle + \frac{2 - \alpha}{2\gamma}\| u^{t+1} \|^{2}\\
& ~~~  - \frac{2 - \alpha}{\gamma}\left\langle u^{t}, v^{t} - u^{t} \right\rangle - \frac{2 - \alpha}{2\gamma} \| u^{t} \|^{2}\\
&= -\frac{1}{2} \left( \frac{1}{\gamma} - l  \right)  \| u^{t+1} - u^{t} \|^{2} + \frac{2 - \alpha}{\gamma} \left\langle u^{t+1} - u^{t}, v^{t} - u^{t+1} \right\rangle\\
& ~~~  + \frac{2 - \alpha}{\gamma}\left\langle u^{t}, u^{t} - u^{t+1} \right\rangle + \frac{2 - \alpha}{2\gamma} \| u^{t+1} \|^{2} - \frac{2 - \alpha}{2\gamma}\| u^{t} \|^{2}\\
&\leq -\frac{1}{2} \left( \frac{1}{\gamma} - l \right)  \| u^{t+1} - u^{t} \|^{2} + \frac{2 - \alpha}{2\gamma}\| u^{t+1} - u^{t} \|^{2} + \frac{2 - \alpha}{2\gamma}\| v^{t} - u^{t+1} \|^{2}\\
& ~~~  + \frac{2 - \alpha}{\gamma}\left\langle u^{t}, u^{t} - u^{t+1} \right\rangle + \frac{2 - \alpha}{2\gamma} \| u^{t+1} \|^{2} - \frac{2 - \alpha}{2\gamma}\| u^{t} \|^{2}\\ 
&\leq \left[ -\frac{1}{2} \left( \frac{1}{\gamma} - l \right)  + \frac{2 - \alpha}{\gamma} \right] \| u^{t+1} - u^{t} \|^{2} + \frac{2 - \alpha}{2\gamma}\| v^{t} - u^{t+1} \|^{2},
\endaligned
\end{equation}
where we have used the elementary fact $2 \left\langle u^{t}, u^{t} - u^{t+1} \right\rangle + \| u^{t+1} \|^{2} - \| u^{t} \|^{2} = \| u^{t} - u^{t+1} \|^{2}$ in the last inequality.   By using \eqref{ineq2} and Lemma \ref{opt-con39}, we obtain further

\begin{equation}\label{u3}
\aligned
&\mathcal{M}_{\gamma}( u^{t+1}, v^{t}, x^{t} ) - \mathcal{M}_{\gamma}( u^{t}, v^{t}, x^{t} )\\
& \leq \left[ -\frac{1}{2} \left( \frac{1}{\gamma} - l  \right)  + \frac{2 - \alpha}{\gamma} \right] \| u^{t+1} - u^{t} \|^{2} \\
&~~~~ + \frac{2 - \alpha}{2\gamma}\bigg[ ( -1 + 2\gamma l )\| u^{t+1} - u^{t} \|^{2} + \| x^{t} - x^{t-1} \|^{2} \bigg]\\
& \leq \left[ -\frac{1}{2} \left( \frac{1}{\gamma} - l \right)  + \frac{ 2 - \alpha}{\gamma}  + \frac{2 - \alpha}{2\gamma} ( -1 + 2 \gamma l )  + \frac{2 - \alpha}{2\gamma}( 1 + \gamma L)^{2} \right]  \| u^{t+1} - u^{t} \|^{2}. 
\endaligned
\end{equation}

Summing up \eqref{x}, \eqref{v2} and \eqref{u3} we obtain
\begin{equation}\label{endeq}
\aligned
&\mathcal{M}_{\gamma}( u^{t+1}, v^{t+1}, x^{t+1} ) - \mathcal{M}_{\gamma}( u^{t}, v^{t}, x^{t} )\\
&\leq \frac{1}{\gamma} \left[ \frac{4 - \alpha}{2} ( 1 + \gamma L)^{2} + \frac{ 8 - 2\alpha +1}{2} \gamma l - \frac{1 + \alpha}{2} \right] \| u^{t+1} - u^{t} \|^{2}\\
& =:-A \| u^{t+1} - u^{t} \|^{2}. 
\endaligned
\end{equation}

Notice the constant $A>0$  by the choice of $\gamma$ and $\alpha$. Therefore, ${\mathcal{M}_{\gamma}( u^{t}, v^{t}, x^{t} )}_{t \geq 1}$ is nonincreasing.

Summing  \eqref{endeq}  from $t = 1$ to $N - 1 \geq 1$, we  get
\begin{equation}\label{sum}
\mathcal{M}_{\gamma} ( u^{N}, v^{N}, x^{N} ) - \mathcal{M}_{\gamma} ( u^{1}, v^{1}, x^{1} ) \leq -A \sum_{t=1}^{N} \| u^{t+1} - u^{t} \|^{2}.
\end{equation}
Therefore, if there exists a cluster point $( u^{*}, v^{*}, x^{*} )$ with a convergent subsequence $lim_{ j \to \infty} ( u^{t_{j}}, v^{t_{j}}, x^{t_{j}} ) = (u^{*}, v^{*}, x^{*} )$, since $\mathcal{M}_{\gamma} $ is lower semi-continuity function and $f$, $g$ are both proper functions, then taking limit as $j \to \infty$ with $N = t_{j}$ in \eqref{sum}, we obtain 
\begin{equation}\label{sumineq}
- \infty < \mathcal{M}_{\gamma}( u^{*}, v^{*}, x^{*} ) - \mathcal{M}_{\gamma}( u^{1}, v^{1}, x^{1} ) \leq -A  \sum_{t=1}^{\infty} \| u^{t+1} - u^{t} \|^{2}. 
\end{equation}
This implies immediately that $lim_{t \to \infty} \| u^{t+1} - u^{t} \|^{2} = 0$.  From Lemma \ref{opt-con39}  we conclude that \eqref{lim} holds.  Furthermore, using the third relation in \eqref{alg}, we get further that $lim_{t \to \infty}\| v^{t+1} - v^{t} \| = 0$. Hence, if $( u^{*}, v^{*}, x^{*} )$ is a cluster point of $\{ ( u^{t}, v^{t}, x^{t} )\}_{t \geq 1}$, that is, the latter has a subsequence $\{ ( u^{t_{j}}, v^{t_{j}}, x^{t_{j}} )\}$ fulfilling $\{ ( u^{t_{j}}, v^{t_{j}}, x^{t_{j}} )\} \to ( u^{*}, v^{*}, x^{*} )$ as $j \to \infty$, then
\begin{equation}\label{sublim}
lim_{j \to \infty} ( u^{t_{j}}, v^{t_{j}}, x^{t_{j}} ) = lim_{j \to \infty} ( u^{t_{j - 1}}, v^{t_{j - 1}}, x^{t_{j - 1}} ) = ( u^{*}, v^{*}, x^{*} ). 
\end{equation}
Since  $v^{t}$ is a minimizer of the second relation in \eqref{alg}, we obtain
\begin{equation}\label{gmin}
g( v^{t} ) + \frac{1}{2\gamma} \| \alpha u^{t} - v^{t} - x^{t-1} \|^{2} \leq g( v^{*} ) + \frac{1}{2\gamma} \| \alpha u^{t} - v^{*} - x^{t-1} \|^{2}. 
\end{equation}
Taking  limit along the convergent subsequence and using \eqref{sublim} yields
\begin{equation}\label{supineq}
\lim \sup_{j \to \infty} g( v^{t_{j}} ) \leq g( v^{*} )
\end{equation}
On the other hand, we have $\lim\inf_{j \to \infty} g( v^{t_{j}} ) \geq g( v^{*} )$ because of the lower semicontinuity of $g$.  Hence
\begin{equation}\label{mmm}
\lim_{j \to \infty} g( v^{t_{j}} ) = g( v^{*} ). 
\end{equation}

Summing \eqref{opt01} and \eqref{opt02} and then taking limit along the convergent subsequence $\{(u^{t_j}, v^{t_j}, x^{t_j})\}$, and  applying \eqref{mmm}, \eqref{lim} and \eqref{subdiffproperty}, the conclusion of the theorem follows immediately.
\end{proof}

We know from the previous theorem that the global subsequence convergence of \cref{alg:GDR} is based on the assumption that there is a cluster point of the sequence. Next we give a sufficient condition to guarantee the boundedness of the sequence generated, and so the existence of cluster  point.
\begin{theorem}\label{the3}$($Boundedness of the sequence generated from the PDR splitting method$)$ Assume that the functions $f$ and $g$ satisfy \cref{ass1}. Let $\alpha \in (\frac{3}{2}, 2] $ and  $\gamma$ satisfy \eqref{gammcond}. Suppose that both $f$ and $g$ are bounded from below, and that at least one of them is coercive. Then the sequence $\{ ( u^{t}, v^{t}, x^{t} )\}$ generated by \eqref{alg} is bounded.
\end{theorem}
\begin{proof} Since $f$ is bounded from below, that is, there exits $\zeta^{*} > -\infty$ such that for any $x$,  we have
\begin{equation}\label{ineq31}
\zeta^{*} \leq f\left(  x - \frac{1}{L} \nabla f(x) \right) \leq f(x) - \frac{1}{2L} \| \nabla f(x) \|^{2}, 
\end{equation}  
where the second inequality follows  from the descent lemma in \cite{BC}. Furthermore, based on the choice of $\gamma$ and $\alpha$, we have from \cref{the1} that for all $ t \geq 1$,
\begin{equation}\label{ineq32}
\mathcal{M}_{\gamma} ( u^{t}, v^{t}, x^{t} ) \leq \mathcal{M}_{\gamma} ( u^{1}, v^{1}, x^{1} )
\end{equation}
Using  \eqref{meritfun58}, we have for $t \geq 1$ that
\begin{equation}\label{ineq33}
\aligned
\mathcal{M}_{\gamma} ( u^{t}, v^{t}, x^{t} ) & = f(u^{t}) + g(v^{t}) - \frac{1}{2\gamma} \| x^{t} - v^{t} \|^{2} + \frac{1}{2\gamma} \| x^{t} - u^{t} \|^{2} \\
&~~~~  + \frac{1}{\gamma} \left\langle (2-\alpha)u^{t}, v^{t} - u^{t} \right\rangle + \frac{2 - \alpha}{2\gamma}\| u^{t} \|^{2}\\
& = f(u^{t}) + g(v^{t}) - \frac{1}{2\gamma} \| x^{t-1} - u^{t} \|^{2} + \frac{1}{2\gamma} \| x^{t} - u^{t} \|^{2} \\
&~~~~  + \frac{2 - \alpha}{\gamma} \left\langle u^{t}, v^{t} - u^{t} \right\rangle + \frac{2 - \alpha}{2\gamma}\| u^{t} \|^{2},
\endaligned
\end{equation}
where the last equality follows from the third relation in \eqref{alg}. Next, we  see  from \eqref{opt01} for any $t \geq 1$ that  
\begin{equation}\label{ineq34}
0 = \nabla f(u^{t}) + \frac{1}{\gamma} ( u^{t} - x^{t-1} ), 
\end{equation}
which implies that  $ \| x^{t-1} - u^{t} \|^{2} = \gamma^{2} \| \nabla f(u^{t})\|^{2}$.  Further, because of \eqref{gammcond}, we can choose $\alpha \sim 2, \mu \in (0, 1)$ such that $-(\frac{1}{2\gamma} + \frac{2 - \alpha}{\gamma}) \gamma^{2} \geq -\frac{1 - \mu}{2L}$. Combining these with \eqref{ineq32} and \eqref{ineq33}, we  obtain that
\begin{equation}\label{ineq35}
\aligned
&\mathcal{M}_{\gamma}( u^{1}, v^{1}, x^{1} ) \geq \mathcal{M}_{\gamma}( u^{t}, v^{t}, x^{t} )\\
&= f(u^{t}) + g(v^{t}) - \frac{1}{2\gamma} \| x^{t-1} - u^{t} \|^{2} + \frac{1}{2\gamma} \| x^{t} - u^{t} \|^{2} + \frac{2 - \alpha}{\gamma} \left[ \left\langle u^{t}, v^{t} - u^{t} \right\rangle +\frac{\| u^{t} \|^{2}}{2} \right] \\
&\geq f(u^{t}) + g(v^{t}) - \frac{1}{2\gamma} \| x^{t-1} - u^{t} \|^{2} + \frac{1}{2\gamma} \| x^{t} - u^{t} \|^{2} - \frac{2 - \alpha}{2\gamma} \| v^{t} - u^{t} \|^{2}\\
&\geq f(u^{t}) + g(v^{t}) - \frac{1}{2\gamma} \left( \| x^{t-1} - u^{t} \|^{2} - \| x^{t} - u^{t} \|^{2} \right) - \frac{2 - \alpha}{\gamma} \| x^{t} - u^{t} \|^{2}  \\
&~~~~ + \frac{2 - \alpha}{\gamma} \| u^{t} - x^{t-1} \|^{2}  \\
&= f(u^{t}) + g(v^{t}) - ( \frac{1}{2\gamma} + \frac{2 - \alpha}{\gamma} ) \| x^{t-1} - u^{t} \|^{2} + ( \frac{1}{2\gamma} - \frac{2 - \alpha}{\gamma} ) \| x^{t} - u^{t} \|^{2}\\
&\geq f(u^{t}) + g(v^{t}) - ( \frac{1}{2\gamma} + \frac{2 - \alpha}{\gamma} ) \gamma^{2} \| \nabla f(u^{t}) \|^{2} + \frac{2\alpha - 3}{2\gamma} \| x^{t} - u^{t} \|^{2}\\
&= \mu f(u^{t}) + (1 - \mu)f(u^{t}) - \frac{1 - \mu}{2L} \| \nabla f(u^{t}) \|^{2} + \left[ \frac{1 - \mu}{2L} - ( \frac{1}{2\gamma} + \frac{2 - \alpha}{\gamma} ) \gamma^{2} \right] \| \nabla f(u^{t}) \|^{2} \\
&~~~~ + g( v^{t} ) + \frac{2\alpha - 3}{2\gamma} \| x^{t} - u^{t} \|^{2} \\
&\geq \mu f(u^{t}) + (1 - \mu) \zeta^{*} + \left[ \frac{1 - \mu}{2L} - ( \frac{1}{2\gamma} + \frac{2 - \alpha}{\gamma} ) \gamma^{2} \right] \| \nabla f(u^{t}) \|^{2} \\
&~~~~ + g( v^{t} ) + \frac{2\alpha - 3}{2\gamma} \| x^{t} - u^{t} \|^{2}, 
\endaligned
\end{equation}

where the first inequality follows from the Cauchy inequality $ab \leq \frac{a^2}{2} + \frac{b^2}{2}$,  the second inequality follows from the third relation in \eqref{alg},  the third inequality follows from the optimal condition \eqref{opt01} and the last inequality follows from \eqref{ineq31}. \\

First, we assume that $g$ is coercive. Notice from \eqref{ineq35} that $\{ v^{t} \}, \{ \nabla f(u^{t}) \}$ and $\{ x^{t} - u^{t} \}$ are bounded. We also obtain from $\eqref{ineq34}$ that $\{ u^{t} - x^{t-1} \}$ is bounded. This combine with the boundedness of $\{ x^{t} - u^{t} \}$ shows that $\{  x^{t} - x^{t-1} \}$ is also bounded. Using the third relation in \eqref{alg}, we derive that $\{ v^{t} - u^{t} \}$ is bounded. Note that we have proved that $\{v^t\}$ is bounded,  hence $\{u^t\}$ is also bounded.   The boundedness of $\{ x^{t} \}$ follows immediately from the boundedness of $\{ x^{t} - u^{t} \}$. 

Finally,  suppose that $f$ is coercive. Then we know  from \eqref{ineq35} that $\{ u^{t} \}$ and $\{ x^{t} - u^{t} \}$ are bounded. This shows that  $\{ x^{t} \}$ is also bounded. We now obtain easily the boundedness of $\{ v^{t} \}$ by the third relation in \eqref{alg}. The proof is completed. 
\end{proof}

Furthermore, assume additionally that the functions $f$ and $g$ are semi-algebraic, we will show in the next theorem that, if the sequence generated by \eqref{alg} has a cluster point, then it is actually convergent.

\begin{theorem}\label{the2} $($Global convergence of the whole sequence$)$ 
Suppose that $0 < \alpha \leq 2$ and the step-size parameter $\gamma>0$ is chosen so that  \eqref{gammcond} holds.  Suppose additionally that the functions $f$ and $g$ are semi-algebraic. If the sequence $\{(u^t, v^t, x^t)\}$ generated by \eqref{alg} has a cluster point, then the whole sequence $\{(u^t, v^t, x^t)\}$ is convergent.  
\end{theorem} 
\begin{proof}
We split the proof into three steps. 

{ {\bf Step 1.} {\it  There exists $\tau >0$ such that whenever $t \geq 1$, 
$$
\textmd{dist}  (0, \partial \mathcal{M}_\gamma (u^t, v^t, x^t)) \leq \tau \|u^{t+1} - u^t\|. 
$$
}}

We first compute the subdifferential of $\mathcal{M}_\gamma$ at $(u^{t+1}, v^{t+1}, x^{t+1})$. It is not difficult  to obtain that for any $t \geq 0$,  
$$
\nabla_x \mathcal{M}_\gamma \Big( u^{t+1}, v^{t+1}, x^{t+1} \Big) = \frac{1}{\gamma} \Big( v^{t+1} - u^{t+1} \Big)  =\frac{1}{\gamma} \Big( x^{t+1} - x^{t} \Big),  
$$
$$
\aligned
\nabla_u   \mathcal{M}_\gamma \Big( u^{t+1}, v^{t+1}, x^{t+1} \Big) & = \nabla f \Big( u^{t+1} \Big) + \frac{1}{\gamma} \Big( u^{t+1} - x^{t+1} \Big)  + \frac{2 - \alpha}{\gamma} \Big( v^{t+1} - u^{t+1} \Big) \\
& = \frac{\alpha -1}{\gamma} \Big(x^t - x^{t+1} \Big),
\endaligned 
$$
where the first gradient follows by using the definition of $u^{t+1}$, while the second gradient follows by using  \eqref{meritfun58} and the relation \eqref{opt01}. Moreover, for the subdifferential with respect to $v$, from  \eqref{meritfun58} we have 
$$
\aligned
\partial_v \mathcal{M}_\gamma \Big( u^{t+1}, v^{t+1}, x^{t+1} \Big) & = \partial g\Big( v^{t+1} \Big) - \frac{1}{\gamma} \Big( v^{t+1} - x^{t+1} \Big) + \frac{2 - \alpha}{\gamma} u^{t+1}  \\
& = \partial g\Big( v^{t+1} \Big) + \frac{1}{\gamma} \Big(  v^{t+1} + x^t - \alpha u^{t+1} \Big)  \\
& ~~~~ + \frac{1}{\gamma} \Big( 2u^{t+1} - x^{t}- v^{t+1} \Big) - \frac{1}{\gamma} \Big( v^{t+1} - x^{t+1} \Big) \\
& \ni - \frac{2}{\gamma} \Big( v^{t+1} - u^{t+1}\Big) + \frac{1}{\gamma} \Big( x^{t+1} - x^{t} \Big)= - \frac{1}{\gamma} \Big( x^{t+1} - x^{t} \Big), 
\endaligned
$$
where the inclusion follows from the relation \eqref{opt02} and the last equality follows from the definition of $x^{t+1}$. The above relations together with Lemma \ref{opt-con39}  imply that there exists $\tau >0$ (in particular, one may take $\tau=\frac{\alpha+3}{\gamma}$) such that whenever $t \geq 1$, we have 
\begin{equation}\label{dis01}
\textmd{dist}  (0, \partial \mathcal{M}_\gamma (u^t, v^t, x^t)) \leq \tau \|u^{t+1} - u^t\|. 
\end{equation}

{\bf Step 2. {\it The limit $\lim_{t \to \infty}\mathcal{M}_\gamma(u^t, v^t, x^t)$ exists and equals to  $\mathcal{M}_\gamma(u^*, v^*, x^*)$ for any cluster point $(u^*, v^*, x^*)$ of sequence $\{(u^t, v^t, x^t)\}$. 
}}

It follows from \eqref{endeq} that there exists $A >0 $ such that
\begin{equation}\label{AAa}
\mathcal{M}_\gamma(u^t, v^t, x^t) - \mathcal{M}_\gamma(u^{t+1}, v^{t+1}, x^{t+1}) \geq A \|u^{t+1} - u^t\|^2. 
\end{equation}
Hence, $\mathcal{M}_\gamma(u^t, v^t, x^t)$ is nonincreasing. Let $\{(u^{t_i}, v^{t_i}, x^{t_i})\}$ be a convergent subsequence which converges to $(u^*, v^*, x^*)$.  Then, by the lower semicontinuity of $\mathcal{M}_\gamma$, we know that the sequence $\{ \mathcal{M}_\gamma(u^{t_i}, v^{t_i}, x^{t_i}) \}$ is bounded below. This together with the nonincreasing property of $\mathcal{M}_\gamma(u^t, v^t, x^t)$ implies that $\mathcal{M}_\gamma(u^t, v^t, x^t)$ is also bounded below. Therefore,  $\lim_{t \to \infty}\mathcal{M}_\gamma(u^t, v^t, x^t) = \theta^*$ exists.  
We claim that $\theta^*=\mathcal{M}_\gamma(u^*, v^*, x^*)$. 

Indeed, let $\{(u^{t_j}, v^{t_j}, x^{t_j})\}$ be any sequence that converges to $(u^*, v^*, x^*)$. Then by the lower semicontinuity, we have
$$
\liminf_{j \to \infty} \mathcal{M}_\gamma(u^{t_j}, v^{t_j}, x^{t_j}) \geq \mathcal{M}_\gamma(u^*, v^*, x^*). 
$$
Moreover,  similar to \eqref{sublim}, \eqref{gmin} and \eqref{supineq},  we also have 
$$
\limsup_{j \to \infty} \mathcal{M}_\gamma(u^{t_j}, v^{t_j}, x^{t_j}) \leq \mathcal{M}_\gamma(u^*, v^*, x^*). 
$$
Now we easily get $\theta^*=\mathcal{M}_\gamma(u^*, v^*, x^*)$, as claimed. Note that if $\mathcal{M}_\gamma(u^{t_0}, v^{t_0}, x^{t_0}) =\theta^*$ for some $t_0 \geq 1$, then $\mathcal{M}_\gamma(u^{t_0+k}, v^{t_0+k}, x^{t_0+k}) = \mathcal{M}_\gamma(u^{t_0}, v^{t_0}, x^{t_0})$ for all $k\geq 0$ since the sequence is nonincreasing. Then from \eqref{AAa}, we have $u^{t_0+k}=u^{t_0}$ for all $k\geq 0$. By \eqref{ineq4}, we see that $x^{t_0+k}=x^{t_0}$ for all $k\geq 0$. These together with the third relation in \eqref{alg}  show that we also have $v^{t_0 +k} = v^{t_0 +1}$ for all $k\geq 1$. Thus, the sequence $(u^t, v^t, x^t)$ remains constant  starting with the $(t_0+1)$st iteration.   Hence, the theorem holds trivially when this happens. 
Next we  always assume that $\mathcal{M}_\gamma(u^{t}, v^{t}, x^{t})  > \theta^*$ for any $t \geq 1$. 
\\

{\bf Step 3. {\it $\{\| u^{t+1} - u^t\|\}$ is summable and end the proof.
}}
It follows from \cite{ABRS} and the semi-algebraic assumption  that the function
$$
(u, v, x) \longmapsto \mathcal{M}_\gamma (u, v, x)
$$
is a KL function. By the property of KL function (see Definition \ref{KL}), there exist $\eta >0$, a neighborhood $U$ of $(u^*, v^*, x^*)$ and a continuous concave function $\varphi : [0, \eta) \to \mathbb{R}_+$ such that for all $(u, v, x) \in U$ satisfying $\theta^* <  \mathcal{M}_\gamma(u, v, x) < \theta^* + \eta$, we have 
\begin{equation}\label{KL098}
\varphi^\prime ( \mathcal{M}_\gamma(u, v, x) - \theta^*) \textmd{dist} (0, \partial \mathcal{M}_\gamma(u, v, x) ) \geq 1. 
\end{equation}
Since $U$ is an open set, take $\rho >0$ such that 
$$
\mathbf{B}_\rho := \{ (u, v, x) : \|u - u^*\| < \rho, \|v - v^*\| < 2\rho, \|x - x^*\| < (2 + \gamma L) \rho\} \subseteq U 
$$
and set $B_\rho:= \{u: \|u - u^*\| <\rho \}$. From \eqref{opt01}, we can get
$$
\|x^t -x^*\| \leq \|x^t - x^{t-1}\| + \|x^{t-1} - x^*\| \leq \|x^t - x^{t-1}\| + (1 + \gamma L) \|u^t - u^*\|.  
$$
 By Theorem \ref{the1}, there exists $N_0 \geq 1$ such that $\|x^t - x^{t-1}\| < \rho$ whenever $t \geq N_0$.  Hence, it follows that $\|x^t  - x^*\| < (2 + \gamma L)\rho$ whenever $u^t \in B_\rho$ and $t \geq N_0$.  
Applying the third relation in \eqref{alg}, we also have that whenever $u^t \in B_\rho$ and for $t \geq N_0$, 
$$
\|v^t - v^*\| \leq \|u^t - u^*\| + \|x^t - x^{t-1}\| < 2\rho.
$$
Therefore, we obtain that if $u^t \in B_\rho$ and $t \geq N_0$, then $(u^t, v^t, x^t) \in \mathbf{B}_\rho \subseteq U$.  Now, by the facts that $(u^*, v^*, x^*) $ is a cluster point, that $\mathcal{M}_\gamma(u^t, v^t, x^t) > \theta^*$ for every $t \geq 1$, and that $\lim_{t\to \infty}\mathcal{M}_\gamma(u^t, v^t, x^t) = \theta^*$, we easily see that there exists $(u^N, v^N, x^N)$ with $N \geq N_0$ such that
\begin{itemize}
\item[(i)] $u^N \in B_\rho$ and $\theta^* < \mathcal{M}_\gamma(u^N, v^N, x^N) < \theta^* +\eta$;

\item[(ii)] $\|u^N - u^*\| + \frac{\tau}{A} \varphi (\mathcal{M}_\gamma (u^N, v^N, x^N) - \theta^*) < \rho$. 
\end{itemize}

Next, we prove that whenever $u^t \in B_\rho$ and $\theta^* < \mathcal{M}_\gamma (u^t, v^t, x^t) < \theta^* + \eta$ for some $t \geq N_0$, we have
\begin{equation}\label{A01}
\|u^{t+1} - u^t\| \leq \frac{\tau}{A} \Big[ \varphi \big( \mathcal{M}_{\gamma}(u^t, v^t, x^t) - \theta^* \big)  - \varphi \big( \mathcal{M}_{\gamma}(u^{t+1}, v^{t+1}, x^{t+1}) - \theta^* \big)  \Big].
\end{equation}
Recall that $\{ \mathcal{M}_{\gamma}(u^t, v^t, x^t) \}$ is non-increasing and $\varphi$ is increasing, \eqref{A01} holds obviously if $u^t=u^{t+1}$. Without loss generality, we assume that $u^{t+1} \neq u^t$. Since $u^t \in B_\rho$ and $t \geq N_0$, we have $(u^t, v^t, x^t) \in \mathbf{B}_\rho \subseteq U$. Hence, for $(u^t, v^t, x^t)$, \eqref{KL098} holds. Using \eqref{dis01}, \eqref{AAa}, \eqref{KL098} and the concavity of $\varphi$,  we obtain  that for such $t$, 
$$
\aligned
& \tau \|u^{t+1} - u^t\| \cdot \Big[\varphi \big( \mathcal{M}_\gamma(u^t, v^t, x^t) - \theta^* \big) - \varphi \big( \mathcal{M}_\gamma(u^{t+1}, v^{t+1}, x^{t+1})  - \theta^* \big)  \Big] \\
& \geq \textmd{dist} (0, \partial \mathcal{M}_\gamma(u^t, v^t, x^t)) \cdot \Big[\varphi \big( \mathcal{M}_\gamma(u^t, v^t, x^t) - \theta^* \big) - \varphi \big( \mathcal{M}_\gamma(u^{t+1}, v^{t+1}, x^{t+1})  - \theta^* \big)  \Big] \\
& \geq \textmd{dist} (0, \partial \mathcal{M}_\gamma(u^t, v^t, x^t)) \cdot \varphi^\prime \big(\mathcal{M}_\gamma(u^t, v^t, x^t) - \theta^* \big)  \\
&~~~~ \cdot \Big[ \mathcal{M}_\gamma(u^t, v^t, x^t)  -  \mathcal{M}_\gamma(u^{t+1}, v^{t+1}, x^{t+1}) \Big] \\
& \geq A \|u^{t+1} - u^t\|^2. 
\endaligned
$$
This implies that \eqref{A01} holds immediately. 

We next claim that $u^t \in B_\rho$ for all $t \geq N$.  First, the claim is true whenever $t=N$ by construction. Now, suppose that the claim is true for $t=N, \dots, N+k-1$ for some $k \geq 1$, that is, $u^N, \dots, u^{N+k-1} \in B_\rho$. Note that $\theta^* < \mathcal{M}_\gamma(u^t, v^t, x^t) < \theta^* + \eta$ for all $t \geq N$ by the choice of $N$ and non-increase property of $\{\mathcal{M}_\gamma(u^t, v^t, x^t)\}$. Hence, \eqref{A01} can be used for $t=N, \dots, N+k-1$. Thus, for $t=N+k$, we have 
$$
\aligned
\| u^{N+k} - u^*\| & \leq \|u^N - u^* \| + \sum_{j=1}^k \|u^{N+j} - u^{N+j-1}\| 
\\
&  \leq \|u^N -u^*\| + \frac{\tau}{A}\sum_{j=1}^k \Big[ \varphi \big( \mathcal{M}_{\gamma}(u^{N+j-1}, v^{N+j-1}, x^{N+j-1}) - \theta^* \big) \\
&~~~  - \varphi \big( \mathcal{M}_{\gamma}(u^{N+j}, v^{N+j}, x^{N+j}) - \theta^* \big)  \Big] \\
& \leq \|u^N -u^*\| + \frac{\tau}{A} \varphi \big( \mathcal{M}_{\gamma}(u^{N}, v^{N}, x^{N}) - \theta^* \big) < \rho. 
\endaligned
$$
Hence, $u^{N+k} \in B_\rho$. By induction, we obtain that $u^t \in B_\rho$ for all $t \geq N$. 

Note that  we have shown that  $u^t \in B_\rho$ and $\theta^* < \mathcal{M}_\gamma(u^t, v^t, x^t) < \theta^* + \eta$ for all $t \geq N$. Summing \eqref{A01} from $t=N$ to $M$ and letting $M \to \infty$, we obtain
$$
\sum_{t=N}^\infty \|u^{t+1} - u^t\| \leq \frac{\tau}{A} \varphi \big( \mathcal{M}_{\gamma}(u^N, v^N, x^N) - \theta^* \big)< +\infty.  
$$
This shows that $\{\| u^{t+1} - u^t\|\}$ is summable and hence the whole sequence $\{u^t\}$ converges  to $u^*$. From this and \eqref{opt01} we also obtain the sequence $\{x^t\}$ is convergent. Finally, by the third relation in \eqref{alg}, the convergence of $\{v^t\}$ follows immediately. The proof is completed. 
\end{proof}

\section{ Applications}
\label{sec:examples} 
In this section, we apply the parameterized Douglas-Rachford splitting method \eqref{alg} to solve three important class of nonconvex optimization problems: sparsity constrained least squares problem, feasibility problems and low rank matrix completion, based on our analysis in \cref{sec:convergence}.

\subsection{Sparsity constrained least squares problems}
\label{ex1}
In this subsection, we apply \cref{alg:GDR} to solve the sparsity constrained least squares problem which arises in the area of statistics and machine learning.

Let $A$ be a linear map, $b$ be a vector and $D$ be a nonempty compact set that is not necessarily convex. Then the constrained least squares problem is given by
\begin{equation}\label{lsp}
\min_{u \in D} \frac{1}{2} \| Au - b \|^2.
\end{equation}
As discussed in \cref{sec:convergence}, when the PDR splitting method is applied directly to \eqref{lsp} with $f(u)= \frac{1}{2} \| Au - b \|^2$ and $g=\mathcal{I_{D}}(u)$, the sequence does not converge to the critical point of \eqref{lsp}. However, as an altenative, we can take $f(u) = \frac{1}{2} \| Au - b \|^2$ and $g(u) = \mathcal{I_{D}} - \frac{2 - \alpha}{2\gamma}\| u \|^2$, and  apply the PDR splitting method accordingly, then the sequence generated  converges  to a critical point of \eqref{lsp}. Thus we can give the following algorithm:
 \begin{equation}\label{lsp-alg}
\begin{cases}
u^{t+1} = \arg\min_{U} \left\{ \frac{1}{2}\| A u - b \|^{2} + \frac{1}{2\gamma}\|u - x^{t}\|^{2} \right\},\\
\\

v^{t+1} = \arg\min_{v} \left\{ \mathcal{I}_{D}(v) - \frac{2 - \alpha}{2\gamma}\| v \|^2 + \frac{1}{2\gamma}\|\alpha u^{t+1} - x^{t} - v\|^{2} \right\},\\
\\

x^{t+1} = x^{t} + (v^{t+1} - u^{t+1}).
\end{cases}
\end{equation}
We can easily solve the two subproblems in \eqref{lsp-alg} to obtain the following algorithm:
 \begin{equation}\label{lsp-alg1}
(PDR)~\begin{cases}
u^{t+1} = \left( \gamma A^TA + I \right)^{-1}(\gamma A^Tb + x^t),\\
\\

v^{t+1} = \mathcal{P}_{D} \left(  \frac{\alpha u^{t+1} - x^t}{\alpha - 1}\right),\\
\\

x^{t+1} = x^{t} + (v^{t+1} - u^{t+1}),
\end{cases}
\end{equation}
where $\mathcal{P}_{D}$ is the projector onto the set $D$. We now check the assumptions on $f$ and $g$ in convergence theory of algorithm \eqref{lsp-alg} in  \cref{sec:convergence}:
\begin{itemize}
\item[1.] For the function $f = \frac{1}{2} \| Au - b\|^2$ in the model \eqref{3-model}, we can easily calculate that it has a Lipschitz continuous gradient, and the Lipschitz constant is $\lambda_{max} (A^TA)$. This verifies $(a1)$ in \cref{ass1};
\item[2.] For the function $g = \mathcal{I}_{D}(u) - \frac{2 - \alpha}{2\gamma}\| u \|^2$, clearly the proximal mapping of $g$ is well-defined for all $u$, so the \cref{ass1} $(a2)$ is satisfied; 
\item[3.] In this experiment, we can take $L = \lambda_{max} (A^TA)$ and $l =0$ in \eqref{gammcond}, from this we easily know that $\gamma$ satisfies \eqref{gammcond} when $0< \gamma < \gamma_{0} = (\sqrt{\frac{1 + \alpha}{4 - \alpha}} - 1)*\frac{1}{\lambda_{max}(A^TA)}$. 
\end{itemize}
Next, we compare our algorithm against the DR splitting method in \cite{LP} and the PR spliting method in \cite{LLP}. Notice that, when $\alpha = 2$, the algorithm \eqref{lsp-alg} is just the classical DR splitting method. In \cite{LLP}, the authors used the PR splitting method to solve \eqref{lsp}. Since they needed that the $f$ is a strongly convex function in the convergence analysis of the PR splitting, they set the function $f = \frac{1}{2} \| Au - b \|^2 + \frac{\beta\lambda_{max}(A^TA)}{2} \| u \|^2$ and the function $g = \mathcal{I}_{D}(u) - \frac{\beta\lambda_{max}(A^TA)}{2} \| u \|^2$. Thus the algorithm in \cite{LLP} is as follows:
\begin{equation}\label{lsp-pr}
(PR)~\begin{cases}
u^{t+1} = \left[ (\beta\gamma \lambda_{max}(A^TA) + 1)I + \gamma A^TA \right]^{-1} (\gamma A^Tb + x^t),\\
\\

v^{t+1} = \mathcal{P}_{D} \left(  \frac{2 u^{t+1} - x^t}{1 - \beta \lambda_{max}(A^TA) \gamma}\right),\\
\\

x^{t+1} = x^{t} + 2(v^{t+1} - u^{t+1}).
\end{cases}
\end{equation}

{\bf Set parameters.} In all algorithms (PDR, PR, DR), we choose the same initialization and stop criteria as in \cite{LP, LLP}. All algorithms are initialized at the origin, and the stop criteria is chosen as follows:
\begin{equation}\label{lsp-stop}
\frac{\max \{ \| x^t - x^{t-1} \|, \| y^t - y^{t-1} \|, \| z^t - z^{t-1} \| \}}{\max \{ \|x^{t-1} \|, \| y^{t-1} \|, \| z^{t-1} \|, 1\}} < 10^{-8}.
\end{equation}
For the choice of $\gamma$, as we analyzed before, we should choose $\gamma < \gamma_0$. However, the $\gamma_0$ might be very small in practical computation. Here, we follow a technique used in \cite{LP} and adopt a heuristic for both DR and PDR splitting: \\

{\it We initialize $\gamma = k * \gamma_0$ and update $\gamma$ as $\max \{ \frac{\gamma}{2}, 0.9999\cdot \gamma_0 \}$ whenever $\gamma > \gamma_0$, and the sequence satisfies either $\| y^t - y^{t-1} \| > 1000/t$ or $\| y \|_{\infty} > 1e10$.}\\

\noindent Following a similar discussion see Remark 4 in \cite{LZ}, one can show that this heuristic leads to a bounded sequence which clusters at a stationary point of \eqref{lsp}. Here we choose $k = 50$ for all algorithms and $\alpha = 1.9, 1.8, 1.7$ for PDR splitting. We set the parameters in PR splitting as in \cite{LLP}: set $\beta = 2.2$ and start with $\gamma = 0.93/(\beta\lambda_{max}(A^TA))$.We then update $\gamma$ as $\max\{\gamma/2, 0.9999 \cdot \gamma_1\}$ whenever $\gamma > \gamma_1 := \frac{\beta - 2}{(\beta + 1)^2 \lambda_{max}(A^TA)}$ and the sequence satisfies either $\| y^t - y^{t-1} \| > \frac{1e3}{t}$ or $\| y \|_{\infty} > 1e10$.\\

\begin{table}[htbp] 
\resizebox{\textwidth}{!}{
 \begin{tabular}{cc|cc|cc|cc|cc|cc} 
  \toprule 
  \multicolumn{2}{l}{Data} & \multicolumn{2}{l}{DR} & \multicolumn{2}{l}{PR} & \multicolumn{2}{l}{1.9-DR} & \multicolumn{2}{l}{1.8-DR} & \multicolumn{2}{l}{1.7-DR}\\
\cmidrule(r){1-2}   \cmidrule(r){3-4}  \cmidrule(r){5-6}  \cmidrule(r){7-8} \cmidrule(r){9-10} \cmidrule(r){11-12} 
 m & n & iter & fval & iter & fval & iter & fval & iter & fval  & iter & fval\\ 
 \midrule 
 \specialrule{0em}{1pt}{1pt}
 100 &  4000 &  274 & 7.73e-02   &  324 & 3.17e-01  & 230 & 1.19e-01 &   204 & 2.03e-01 & 279 & 3.73e-01\\ 
 \specialrule{0em}{1pt}{1pt}
 100 &  5000  &    291 & 2.06e-01&    370 & 4.95e-01   &    250 & 3.28e-01  &  213 & 4.52e-01   &    338 & 5.08e-01\\
 \specialrule{0em}{1pt}{1pt}
 100 &  6000 &    301 & 2.53e-01 &    436 & 4.76e-01  &    258 & 3.13e-01 & 254 & 4.54e-01  &    423 & 5.12e-01\\
 \specialrule{0em}{1pt}{1pt}
 \cline{1-12}
 \specialrule{0em}{1pt}{1pt}
200 &  4000 &    217 &  9.20e-03   &    185 & 7.59e-02  &    187 & 1.42e-02  &    160 & 1.39e-02 & 155 & 9.32e-02\\
\specialrule{0em}{1pt}{1pt}
200 &  5000  &    234 & 9.10e-03 &    224 & 2.06e-01 &    200 & 9.02e-03  &  173 & 6.91e-02   &    184 & 1.65e-01\\
\specialrule{0em}{1pt}{1pt}
200 &  6000 &    250 & 8.94e-03  &    281 & 1.77e-01  &    217 & 3.41e-02 & 196 & 9.27e-02 &    235 & 2.42e-01\\
\specialrule{0em}{1pt}{1pt}
 \cline{1-12}
 \specialrule{0em}{1pt}{1pt}
 300 &  4000 &    184 & 1.31e-02   &    123 & 1.39e-02  &    158 & 1.32e-02  &    132 & 1.33e-02  & 118 & 1.30e-02\\
 \specialrule{0em}{1pt}{1pt}
 300 &  5000 &    201 & 1.35e-02 &    150 & 1.42e-02 &    171 & 1.36e-02  &    144 & 1.39e-02  & 133 & 1.49e-02\\
 \specialrule{0em}{1pt}{1pt}
 300 &  6000  &    215 & 1.32e-02 &    187 & 1.44e-02  &    183 & 1.33e-02  &    156 & 1.35e-02  &  156 & 1.44e-02\\
 \specialrule{0em}{1pt}{1pt}
 \cline{1-12}
 \specialrule{0em}{1pt}{1pt}
 400 &  4000  &    166 & 1.75e-02  &  91 & 1.79e-02  &    141 & 1.76e-02  &    118 & 1.77e-02  &     99 & 1.74e-02\\
 \specialrule{0em}{1pt}{1pt}
 400 &  5000  &    179 & 1.75e-02  &    115 & 1.83e-02  &    153 & 1.75e-02  &    129 & 1.77e-02  &  112 & 1.74e-02\\
 \specialrule{0em}{1pt}{1pt}
 400 &  6000  &    194 & 1.77e-02 &    140 & 1.85e-02   &    165 & 1.78e-02 &    138 & 1.80e-02  &  126 & 1.82e-02\\
 \specialrule{0em}{1pt}{1pt}
  \cline{1-12}
  \specialrule{0em}{1pt}{1pt}
 500 &  4000 &    160 & 2.23e-02 &   75 & 2.27e-02  &    131 & 2.24e-02  &    108 & 2.25e-02  &     89 & 2.27e-02\\
 \specialrule{0em}{1pt}{1pt}
 500 &  5000 &    166 & 2.17e-02   & 92 & 2.22e-02  &    140 & 2.17e-02  &    118 & 2.18e-02 &     98 & 2.20e-02\\
 \specialrule{0em}{1pt}{1pt}
 500 &  6000  &    178 &  2.22e-02 &    112 & 2.27e-02  &    151 & 2.22e-02  &    127 & 2.24e-02  & 110 & 2.25e-02\\
 \specialrule{0em}{1pt}{1pt}
  \bottomrule 
 \end{tabular} 
 }
 \caption{\label{tab:lsp}Comparing DR, PR and $\alpha$-DR splitting method for constrained least squares problem on random instances} 
\end{table}

{\bf Simulated data.}  We randomly generate $m \times n$ matrix $A$, a random sparsity vector $\tilde{x} \in \mathbb{R}^n$ with $r = \lceil \frac{m}{10} \rceil$ non-zero entries and a noise vector $\epsilon \in \mathbb{R}^m$. Then we get the vector $b \in \mathbb{R}^m$ by $b = A\tilde{x} + noiselevel*\epsilon$, where we take the $noiselevel = 0.01$. We set $D = \{ x \in \mathbb{R}^n : \| x \|_0 \leq r, \| x \|_{\infty} \leq 10^6 \}$. We take $m \in \{100, 200, 300, 400, 500 \}$ and $n \in \{4000, 5000, 6000\}$, respectively. The computational results are displayed in \cref{tab:lsp}. All of these quantities are averaged over 50 runs. In \cref{tab:lsp}, we report the number of iterations (iter) and the final value of the objective function (fval).  We can see that the PDR splitting method is always faster than the DR splitting method and the final function value is comparable with DR splitting. For example, when $m= 300,400$ and $500$, PDR splitting with $\alpha = 1.7$ is faster than DR splitting while the final function value is similar to DR method. Compared with PR splitting, when $m$ is relatively small, such as $m = 100$ or $200$, PDR splitting with $\alpha = 1.9$ or $1.8$ is faster than PR method and the final function value is also smaller than PR's. When $m = \in \{300, 400, 500 \}$, we see that PR method is faster than PDR and their final function value are similar. \\

\subsection{Nonconvex feasibility problem}
Another important class of optimization problems is the feasibility problem, such as see \cite{ABS,BB,BB1,HL}, which search a point in the intersection of two nonempty sets. Let $C$ be an nonempty closed convex set and $D$ be an nonempty compact set. Then the feasibility problem is to find a point in $ C \cap D$. It is well known that this problem can be modeled by the following optimization: 
\begin{equation}\label{fsp}
\begin{aligned}
&\min_{u} ~~~~\frac{1}{2} d_{C}^2 (u)\\
& s.t.~~~~~~u \in D,
\end{aligned}
\end{equation}
where $d_{C}^2 (u) := \inf_{y \in C} \| y - u \|^2$ is the distance function. Obviously, we can see that when $C \cap D \neq \emptyset$, the optimization problem \eqref{fsp} has a zero optimal value.\\

Similar to \cref{ex1}, we take $f(u) = \frac{1}{2} d_{C}^2 (u)$ and $g(u) = \mathcal{I}_{D}(u) - \frac{2 - \alpha}{2\gamma}\| u \|^2$ in this experiment. Then applying the PDR splitting to solving problem \eqref{fsp} gives the following algorithm:
 \begin{equation}\label{fsp-alg}
\begin{cases}
u^{t+1} = \arg\min_{u} \left\{ \frac{1}{2} d_{C}^{2}(u) + \frac{1}{2\gamma}\|u - x^{t}\|^{2} \right\},\\
\\

v^{t+1} = \arg\min_{v} \left\{ \mathcal{I}_{D}(v) - \frac{2 - \alpha}{2\gamma}\| v \|^2 + \frac{1}{2\gamma}\|\alpha u^{t+1} - x^{t} - v\|^{2} \right\},\\
\\

x^{t+1} = x^{t} + (v^{t+1} - u^{t+1}).
\end{cases}
\end{equation}
Furthermore, we solve the two subproblems of the above algorithm and get
 \begin{equation}\label{fsp-alg1}
(PDR)~\begin{cases}
u^{t+1} = \frac{1}{1 + \gamma} (x^t + \gamma \mathcal{P}_{C}(x^t)),\\
\\

v^{t+1} = \mathcal{P}_{D} \left(  \frac{\alpha u^{t+1} - x^t}{\alpha - 1}\right),\\
\\

x^{t+1} = x^{t} + (v^{t+1} - u^{t+1}).
\end{cases}
\end{equation}
Next we check the assumptions on $f$ and $g$ in convergence theory of the algorithm \eqref{fsp-alg} in  \cref{sec:convergence}:
\begin{itemize}
\item[1.] Since $C$ is a closed and convex, we know that $f = \inf_{y \in C} \| y - u \|^2$ is smooth with a Lipschitz continuous gradient whose Lipschitz continuity modulus $L$ is 1; see, for example, Corollary 12.30 in \cite{BC}. This verifies the \cref{ass1} $(a1)$;
\item[2.] For the function $g = \mathcal{I}_{D}(u) - \frac{2 - \alpha}{2\gamma}\| u \|^2$, the proximal of $g$ exists, hence the \cref{ass1} $(a2)$ is satisfied;
\item[3.] In this experiment, we can take $L=1$ and $l=0$ in \eqref{gammcond}, therefore we easily compute that the $\gamma$ satisfies  \eqref{gammcond} when $0< \gamma < \gamma_{0} = \sqrt{\frac{1 + \alpha}{4 - \alpha}} - 1$.\\
\end{itemize}

We now compare the PDR splitting with the DR, PR splitting and alternating projection methods. Note that when $\alpha = 2$, the \eqref{fsp-alg} is the DR splitting method to solving the feasibility problem. In order to ensure the convergence, in \cite{LLP} the authors set $f (u) = \inf_{y \in C} \| y - u \|^2 + \frac{\beta}{2} \| u \|^2$ and $g(u) = \mathcal{I}_{D} (u) - \frac{\beta}{2} \| u \|^2$ to solve the problem \eqref{fsp} by the PR splitting method, thus their algorithm can be given as follows:
\begin{equation}\label{fsp-pr}
(PR)~\begin{cases}
u^{t+1} = \frac{\gamma \mathcal{P}_{C} \left( \frac{x^t}{1 + \beta\gamma}\right) + x^t}{(1+\beta)\gamma + 1},\\
\\
v^{t+1} = \mathcal{P}_{D} \left( \frac{2u^{t+1} - x^t}{1 - \beta \gamma} \right),\\
\\
x^{t+1} = x^t + 2 (z^{t+1} - u^{t+1}).
\end{cases}
\end{equation}
The alternating projection method to solve problem \eqref{fsp} is described as follows: given $x^0$, updating
\begin{equation}\label{fsp-alt}
(ALT)~~x^{t+1} \in {\arg\min}_{\{\| x \|_0 \leq r, \| x \|_{\infty} \leq 10^6\}} \big\{ \| x - \left( x^t + A^{\dagger}(b - Ax^t) \right) \| \big\}.\\
\end{equation}

{\bf Simulation data.} 
We consider the problem of finding an $r$-sparse solution of a randomly generated linear system $Ax = b$. In detail, we set $C = \{ x \in \mathbb{R}^n : Ax = b \}$ and $D = \{ x \in \mathbb{R}^n : \| x \|_0 \leq r, \| x \|_{\infty} \leq 10^6 \}$. We randomly generate an $m \times n$ matrix $A$ with i.i.d. standard Gaussian entries. We generate a sparse vector $\tilde{x} \in \mathbb{R}^n$ with $r = \lceil \frac{m}{5} \rceil$ non-zero entries randomly. We then project $\tilde{x}$ onto $[10^{-6}, 10^6]^n$, hence $\tilde{x} \in D$. Finally, set $b = A\tilde{x}$. Then $C \cap D$ is nonempty because it contains at least $\tilde{x}$. This implies that the globally optimal value of the objective function is zero.\\

{\bf Set parameters.} In all algorithms, we still choose the same stop criteria as same as \eqref{lsp-stop}. All algorithms are initialized at the origin. For the choice of $\gamma$, as we analyzed before, we should choose $\gamma < \gamma_0$. However, the $\gamma_0$ might be very small in practical computation. Here, we follow a technique used in \cite{LP} and adopt a heuristic for both DR and PDR splitting: \\

{\it We initialize $\gamma = k * \gamma_0$ and update $\gamma$ as $\max \{ \frac{\gamma}{2}, 0.9999\cdot \gamma_0 \}$ whenever $\gamma > \gamma_0$, and the sequence satisfies either $\| y^t - y^{t-1} \| > 1000/t$ or $\| y \|_{\infty} > 1e10$.}\\

Here we choose $k = 150$ for all algorithms and $\alpha = 1.7$ for PDR splitting. We set the parameters in PR splitting as in \cite{LLP}: set $\beta = 2.2$ and start with $\gamma = 0.93/\beta$.We then update $\gamma$ as $\max\{\gamma/2, 0.9999 \cdot \gamma_1\}$ whenever $\gamma > \gamma_1 := \frac{\beta - 2}{(\beta + 1)^2}$ and the sequence satisfies either $\| y^t - y^{t-1} \| > \frac{1e3}{t}$ or $\| y \|_{\infty} > 1e10$.\\

We generate 50 random instances for each pair of $(m,n)$, where $m \in \{ 100, 200, $ $300, 400, 500 \}$ and $n \in \{ 4000, 5000, 6000 \}.$  We report our computational results in \cref{tab:fsp1} and \cref{tab:fsp2} , where all results averaged over the 50 instances. We present the runtime (time), number of iteration (iter), the largest and smallest function values at termination ($fval_{max},fval_{min}$), and also the number of successes (succ) and failures (fail) in identifying a sparse solution of the linear system. Here we declare a success if the function value at termination is below $10^{-12}$ and a failure if the value at termination is above $10^{-12}$. We also give the average runtime ($tim_{s}$) and the number of iterations ($iter_{s}$) for successful instances.

\begin{table}[htbp] 
\resizebox{\textwidth}{!}{
 \begin{tabular}{cc|cccc|cccc|cccc|cccc} 
  \toprule 
  \multicolumn{2}{l}{Data} & \multicolumn{4}{l}{DR} & \multicolumn{4}{l}{PR} & \multicolumn{4}{l}{Alt-Pro} & \multicolumn{4}{l}{1.7-$DR_{150}$}\\
\cmidrule(r){1-2}   \cmidrule(r){3-6}  \cmidrule(r){7-10}  \cmidrule(r){11-14} \cmidrule(r){15-18} 
 m & n & tim(s) & iter & $tim_{s}$ & $iter_{s}$ & tim(s) & iter & $tim_{s}$ & $iter_{s}$ & tim(s) & iter & $tim_{s}$ & $iter_{s}$ & tim(s) & iter & $tim_{s}$ & $iter_{s}$\\ 
 \midrule 
 \specialrule{0em}{1pt}{1pt}
 100 &  4000 & 1.02 &   1967 &  0.87  & 1544 & 0.16 &  307  & - & - &    0.78 &   1694 &- &- &0.19 &    372 & 0.21 &444\\ 
 \specialrule{0em}{1pt}{1pt}
 100 &  5000 &    1.69 &   2599 &  1.32 &   1685 &   0.25 &    381  & - & -  & 1.16 &   1978 & - & - & 0.24 &    373 & - & -   \\
 \specialrule{0em}{1pt}{1pt}
 100 &  6000 &   1.58 &   2046  & 0.97  & 1257 &    0.34 &    429  & 0.47  &   646  &  1.68 &   2350 & - & - &  0.29 &    380 &  0.33 & 429\\
 \specialrule{0em}{1pt}{1pt}
 \cline{1-18}
 \specialrule{0em}{1pt}{1pt}
200 &  4000 &  0.80 &    836  & 0.80    &  836 & 0.20 &    199   & 0.23  & 232  &1.00 &   1076 & - & -
&  0.38 &    384 & 0.38 & 385\\
\specialrule{0em}{1pt}{1pt}
200 &  5000 & 1.13 &   1080 & 1.13 &1080 & 0.25 &    232 & 0.30 &  280 & 1.23 &   1223&-&- &0.39 &    371&  0.43 &397 \\
\specialrule{0em}{1pt}{1pt}
200 &  6000 & 1.60 &   1279 & 1.49  & 1197 &    0.34 &    262 & 0.43 & 334  & 1.84 &   1510 & 2.97  & 2675 &  0.47 &    376 &  0.50 &  408\\
\specialrule{0em}{1pt}{1pt}
 \cline{1-18}
 \specialrule{0em}{1pt}{1pt}
 300 &  4000 &  0.92 &    600 & 0.92   &600 & 0.23 &    139 & 0.25 & 149 & 1.32 &    872 &1.92 & 1282    & 0.54 &    343&    0.54 & 343 \\
 \specialrule{0em}{1pt}{1pt}
 300 &  5000 &   1.32 &    710 & 1.32 & 710 &  0.34 &    170 & 0.38 & 189 &1.99 &   1068 &  3.18 &    1548 & 0.77 &    402 &   0.76 & 389\\
 \specialrule{0em}{1pt}{1pt}
 300 &  6000 &  1.63 &    812 & 1.63 & 812 &  0.43 &    212  &  0.51 &246 & 2.44 &   1252 &2.88 &    1554 & 0.78 &    383 &    0.77 & 385\\
 \specialrule{0em}{1pt}{1pt}
 \cline{1-18}
 \specialrule{0em}{1pt}{1pt}
 400 &  4000 &    1.17 &    520 &   1.17 &    520 &  0.22 &     93 & 0.22 &     93  & 1.80 &    818 &    1.87  & 854 &  0.66 &    287 &  0.66 &    287\\
 \specialrule{0em}{1pt}{1pt}
 400 &  5000 &  1.48 &    579 & 1.48 &    579 &  0.33 &    125 & 0.34 & 128 & 2.43 &    946&2.91&    1176 & 0.86 &    328 &  0.86 &    328 \\
 \specialrule{0em}{1pt}{1pt}
 400 &  6000 &    1.87 &    646 & 1.87 &    646 &0.49 &    156 &     0.42 &167 &  3.13 &   1108 &     4.30  & 1530 & 1.11 &    374& 1.08  & 361 \\
 \specialrule{0em}{1pt}{1pt}
  \cline{1-18}
  \specialrule{0em}{1pt}{1pt}
 500 &  4000 & 1.59 &    499 &  1.59 &    499 & 1.22 &    381 & 0.21  & 64 & 1.97 &    640 &  1.96 &    636 & 0.79 &    258 & 0.79 &    258 \\
 \specialrule{0em}{1pt}{1pt}
 500 &  5000 &1.66 &    519 &  1.66 &    519 &0.31 &     91  &0.31 &     91 &2.68 &    846 &  2.79&     882 &   0.92 &    285 & 0.92 &    285\\
 \specialrule{0em}{1pt}{1pt}
 500 &  6000 &  1.97 &    556 & 1.97 &    556 &0.45 &    122 & 0.44 &   120& 3.72 &   1071 &4.14 &    1192 & 1.11 &    311 & 1.11 &    311\\
 \specialrule{0em}{1pt}{1pt}
  \bottomrule 
 \end{tabular} 
 }
 \caption{\label{tab:fsp1}Comparing DR, PR, Alt and PDR splitting method for feasibility problem on random instances} 
\end{table}

\begin{table}[htbp] 
\resizebox{\textwidth}{!}{
 \begin{tabular}{cc|cccc|cccc|cccc|cccc} 
  \toprule 
  \multicolumn{2}{l}{Data} & \multicolumn{4}{l}{DR} & \multicolumn{4}{l}{PR} & \multicolumn{4}{l}{Alt-Pro} & \multicolumn{4}{l}{1.7-$DR_{150}$}\\
\cmidrule(r){1-2}   \cmidrule(r){3-6}  \cmidrule(r){7-10}  \cmidrule(r){11-14} \cmidrule(r){15-18} 
 m & n & $f_{max}$ & $f_{min}$ & succ & fail & $f_{max}$ & $f_{min}$ & succ & fail & $f_{max}$ & $f_{min}$ & succ & fail & $f_{max}$ & $f_{min}$ & succ & fail\\ 
 \midrule 
 \specialrule{0em}{1pt}{1pt}
 100 &  4000 & 3e-02 & 6e-17 & 30 & 20 &  6e-02 & 1e-03 &  0 & 50 & 8e-02 & 4e-03 &  0 & 50  & 3e-02 & 2e-16 &  1 & 49 \\ 
 \specialrule{0em}{1pt}{1pt}
 100 &  5000 & 2e-02 & 2e-16 & 18 & 32 &5e-02 & 3e-04 &  0 & 50 &  7e-02 & 5e-03 &  0 & 50 &     4e-02 & 5e-03 &  0 & 50\\
 \specialrule{0em}{1pt}{1pt}
 100 &  6000 & 1e-02 & 1e-16 & 12 & 38 & 5e-02 & 4e-15 &  1 & 49 &5e-02 & 4e-05 &  0 & 50 &     2e-02 & 7e-16 &  1 & 49\\
 \specialrule{0em}{1pt}{1pt}
 \cline{1-18}
 \specialrule{0em}{1pt}{1pt}
200 &  4000 &  2e-15 & 2e-16 & 50 &  0 &  2e-01 & 1e-15 & 20 & 30 &3e-01 & 3e-05 &  0 & 50  &   9e-02 & 4e-16 & 25 & 25\\
\specialrule{0em}{1pt}{1pt}
200 &  5000 &  3e-15 & 2e-16 & 50 &  0  &1e-01 & 1e-15 &  7 & 43 & 2e-01 & 2e-03 &  0 & 50&     1e-01 & 3e-16 &  5 & 45 \\
\specialrule{0em}{1pt}{1pt}
200 &  6000 & 7e-02 & 1e-16 & 43 &  7 & 1e-01 & 3e-15 &  6 & 44 & 2e-01 & 1e-13 &  1 & 49 & 8e-02 & 2e-16 &  4 & 46\\
\specialrule{0em}{1pt}{1pt}
 \cline{1-18}
 \specialrule{0em}{1pt}{1pt}
 300 &  4000 &  3e-15 & 2e-16 & 50 &  0  & 2e-01 & 1e-15 & 32 & 18 & 4e-01 & 6e-14 &  3 & 46  &     2e-15 & 4e-16 & 50 &  0\\ 
 \specialrule{0em}{1pt}{1pt}
 300 &  5000 &    4e-15 & 4e-16 & 50 &  0 & 3e-01 & 1e-15 & 26 & 24 & 3e-01 & 9e-14 &  3 & 45 & 1e-01 & 4e-16 & 40 & 10 \\
 \specialrule{0em}{1pt}{1pt}
 300 &  6000 &  3e-15 & 2e-16 & 50 &  0 &  2e-01 & 2e-15 & 25 & 25 &3e-01 & 1e-13 &  1 & 49 &     1e-01 & 3e-16 & 28 & 22\\
 \specialrule{0em}{1pt}{1pt}
 \cline{1-18}
 \specialrule{0em}{1pt}{1pt}
 400 &  4000 &    2e-15 & 3e-17 & 50 &  0 & 3e-01 & 8e-16 & 49 &  1 &6e-01 & 7e-14 & 30 & 19&     4e-15 & 6e-16 & 50 &  0\\
 \specialrule{0em}{1pt}{1pt}
 400 &  5000 & 3e-15 & 5e-16 & 50 &  0 &3e-01 & 1e-15 & 39 & 11 &4e-01 & 9e-14 & 12 & 36 &     4e-15 & 6e-16 & 50 &  0\\
 \specialrule{0em}{1pt}{1pt}
 400 &  6000 &  4e-15 & 6e-16 & 50 &  0 & 2e-01 & 2e-15 & 33 & 17 & 3e-01 & 1e-13 &  4 & 44 &7e-02 & 8e-16 & 48 &  2  \\
 \specialrule{0em}{1pt}{1pt}
  \cline{1-18}
  \specialrule{0em}{1pt}{1pt}
 500 &  4000 &   1e-16 & 1e-18 & 50 &  0 &3e-01 & 4e-16 & 43 &  7  &4e-01 & 6e-14 & 38 & 10 &     5e-15 & 8e-16 & 50 &  0\\
 \specialrule{0em}{1pt}{1pt}
 500 &  5000 &1e-15 & 4e-17 & 50 &  0 & 2e-15 & 9e-16 & 50 &  0&4e-01 & 8e-14 & 37 & 13 & 4e-15 & 3e-16 & 50 &  0\\
 \specialrule{0em}{1pt}{1pt}
 500 &  6000 &  3e-15 & 3e-16 & 50 &  0 &    2e-01 & 1e-15 & 44 &  6 & 5e-01 & 1e-13 & 22 & 28& 4e-15 & 7e-16 & 50 &  0\\
 \specialrule{0em}{1pt}{1pt}
  \bottomrule 
 \end{tabular} 
 }
 \caption{\label{tab:fsp2}Comparing DR, PR, Alt and PDR splitting method for feasibility problem on random instances} 
\end{table}

From \cref{tab:fsp1} and \cref{tab:fsp2}, we can see that both DR and PDR generally achieve higher success rate compared to Alternating projection method and PR, while PDR is faster than DR for some cases. For example, when $m \in \{400, 500\}$, the success rate of PDR splitting and DR splitting is almost the same, but PDR splitting is much faster than DR method. Although PR splitting is faster than DR and PDR splitting for $m \in \{300, 400, 500\}$, the function value and success rate of PR splitting are worse than DR and PDR methods. Meanwhile, the alternating projection method is not comparable to any of them. We also notice that when $m$ is small, such as $m = 100$ or $200$,  the success rate of the PR and PDR methods  are very similar,  the final function value of PDR is less than that of PR splitting although PR method is faster than PDR method.

\subsection{Low rank matrix completion}
There is a rapidly growing interest in the recovery of an unknown low-rank or approximately low-rank matrix from very limited information. In the following we employ the parameterized  Douglas-Rachford splitting method to resolve this problem and compare our results to the previous ones. \\

{\bf Simulation data.}   We generate $n \times n$ matrices of rank $r$ by sampling two $n \times r$ factors $M_{L}$ and $M_{R}$ independently, each having i.i.d. Gaussian entries, and setting $M = M_{L}M_{R}^{*}$ as suggested in \cite{CR}. The set of observed entries $\Omega$ is sampled uniformly at random among all sets of cardinality $m$. The sampling ratio is defined as $p:=\frac{m}{n^2}$ . We wish to recover a matrix with lowest rank such that its entries are equal to those of $M$ on $\Omega$.  For this purpose, we consider the following optimization problem 
\begin{equation}\label{ex3-opt} 
{\arg\min}_{X} \left\{ \frac{1}{2}\|\mathcal{P}_{\Omega}(X) - \mathcal{P}_{\Omega}(M)\|^{2}  + I_{C(r)}(X) \right\},
\end{equation}
where $C(r) := \{X ~ | ~ \textmd{rank}(X) \leq r\}$, $I_{C(r)}(\cdot)$  denotes the indicator function  of $C(r)$ and $\mathcal{P}_{\Omega}$ is the orthogonal projector onto the span of matrices vanishing outside of $\Omega$ so that the $(i, j)$th componet of $\mathcal{P}_{\Omega}(X)$ is equal to $X_{ij}$ if $(i, j) \in \Omega$ and zero otherwise.  For some important nonconvex optimization problems, such as see \cite{ZH,LZT}, the regularization term has been shown to be very efficient. Hence, here we study problem \eqref{ex3-opt} via adopting two different choices about functions $f$ and $g$, one of which will produce an additional regularization term. \\ 

\noindent {\bf Method 1:  Take $f(X) = \frac{1}{2}\|\mathcal{P}_{\Omega}(X) - \mathcal{P}_{\Omega}(M)\|^{2}$ and $g(X) = I_{C(r)}(X) - \frac{2-\alpha}{2\gamma} \| X \|^2$.} Then applying the PDR splitting to solving problem \eqref{ex3-opt} gives the following algorithm: 
\begin{equation}\label{mcp-alg1}
\begin{cases}
U^{t+1} = \arg\min_{U} \left\{ \frac{1}{2}\|\mathcal{P}_{\Omega}(U) - \mathcal{P}_{\Omega}(M)\|^{2} + \frac{1}{2\gamma}\|U-X^{t}\|^{2} \right\},\\
\\

V^{t+1} = \arg\min_{V} \left\{ I_{C(r)}(V) -\frac{2 - \alpha}{2\gamma} \| V \|^2 + \frac{1}{2\gamma}\|\alpha U^{t+1} - X^{t} - V\|^{2} \right\},\\
\\

X^{t+1} = X^{t} + (V^{t+1} - U^{t+1}).
\end{cases}
\end{equation}
Clearly, when $\alpha = 2$, this algorithm is just the classical DR splitting method. By solving the subproblems of the above algorithm, we have 
\begin{equation} 
(PDR~)\begin{cases}
 U^{t+1}=
\begin{cases}
 \frac{1}{1+\gamma}\left( X_{i, j}^{t} + \gamma M_{i, j}\right),  ~~~ & (i, j) \in \Omega, \\
 X_{i, j}^{t},  ~~~~ & (i, j)\notin \Omega, \\
\end{cases}
\\
\\

V^{t+1} = P_{C(r)} \left( \frac{\alpha U^{t+1} - V^{t}}{\alpha - 1} \right),\\

\\
X^{t+1} = X^{t} + (V^{t+1} - U^{t+1}).
\end{cases}
\end{equation} 
We also verify the assumptions on $f$ and $g$ in convergence theory of the algorithm \eqref{mcp-alg1} in \cref{sec:convergence}:
\begin{itemize}
\item[1.] Since $\mathcal{P}_{\Omega}$ is the orthogonal projection, we can easily know that $f(X) = \frac{1}{2}\|\mathcal{P}_{\Omega}(X) - \mathcal{P}_{\Omega}(M)\|^{2}$ is smooth with a Lipschitz continuous gradient whose Lipschitz continuity modulus $L$ is 1.Therefore, this verifies the \cref{ass1} $(a1)$;
\item[2.] For the function $g(X) = I_{C(r)}(X) - \frac{2-\alpha}{2\gamma} \| X \|^2$, the proximal of $g$ exists, hence the \cref{ass1} $(a2)$ is satisfied;
\item[3.] In this experiment, we can take $L=1$ and $l=0$ in \eqref{gammcond}, therefore we easily compute that the $\gamma$ satisfies  \eqref{gammcond} when $0< \gamma < \gamma_{0} = \sqrt{\frac{1 + \alpha}{4 - \alpha}} - 1$.
\end{itemize}
From the above analysis and the convergence theory in \cref{sec:convergence}, we know that the sequence generated by \eqref{mcp-alg1} converges to the critical point of \eqref{ex3-opt}. Next, we compare our algorithm against the DR splitting method in \cite{LP} and SVP method in \cite{JMD}. We recall the SVP \cite{JMD} algorithm to solve problem \eqref{ex3-opt}:
\begin{equation} 
(SVP) ~ \begin{cases}
Y^{t+1} = X^t - \eta_t \mathcal{P}_\Omega^T (\mathcal{P}_\Omega(X^t) - b),~~~~~~~~~~~~~~~~~~~~~~~\\
X^{t+1} = P_{C(r)}( Y^{t+1} ).
\end{cases}
\end{equation} 
{\bf Set parameters.} All algorithms (PDR, DR and SVP) are initialized at the origin and terminated when:
\begin{equation}\label{mcp-stop}
\frac{\| \mathcal{P}_{\Omega} ( X^{t} - M) \|_{F}}{\| \mathcal{P}_{\Omega} ( M ) \|_{F}} < 10^{-4},
\end{equation}
where $\|\cdot\|_F$ represents  the Frobenius norm. We compute relative error for two algorithms by the following way:
\begin{equation}\label{mcp-error}
\textmd{relative} ~ \textmd{error} = \frac{\| X^{\textmd{opt}} - M \|_{F}}{\| M \|_{F}}.  
\end{equation} 
As before, we also adopt a heuristic for both algorithms to update $\gamma$: \\

{\it We initialize $\gamma = k * \gamma_0$ and update $\gamma$ as $\max \{ \frac{\gamma}{2}, 0.9999\cdot \gamma_0 \}$ whenever $\gamma > \gamma_0$, and the sequence satisfies either $\| U^t - U^{t-1} \|_{F}/n > 1000/t$ or $\| U \|_{\infty} > 1e10$.}\\

\noindent Following a similar discussion see Remark 4 in \cite{LZ}, one can show that this heuristic leads to a bounded sequence which clusters at a stationary point of \eqref{lsp}. Here we choose $k = 150$ for both algorithms and consider four different values of $\alpha = 1.9, 1.8, 1.7, 1.6$ for PDR splitting. For the SVP method, the parameter $\eta_t$ is set to $\eta_t = \frac{1}{p\sqrt{t}}$ as in \cite{JMD}. We take the low sampling ratio $p=0.08$ and recover respectively the matrix of rank =10 and 30 in the different sizes $n=3000, 5000, 8000$ and 10000.  Our computational results are displayed in \cref{tab-mcp1} and \cref{tab-mcp2}. All of these quantities are averaged over five runs.  \\

\begin{table}[htbp]\footnotesize
\centering
\begin{tabular}{|c|c|cccccc|} 
\hline
\multicolumn{1}{|c|}{rank}&\multicolumn{1}{|c|}{size}
&\multicolumn{6}{|c|}{ Average runtime(s) / iterations} \\
\hline
&&SVP&DR&$\alpha = 1.9$&$\alpha = 1.8$&$\alpha = 1.7$&$\alpha = 1.6$\\ \cline{2-8}
&3000&176/618 & 111/207 & 100/188  & 84/163 &67/130 & \textbf{43/84}  \\ \cline{2-8}

rank=10 &5000& 421/526 & 257/191&241/174& 210/153 & 180/126 & \textbf{126/83}  \\ \cline{2-8}

 &8000 & 973/474&920/180 & 788/165  & 619/148 & 413/121 & \textbf{314/82} \\ \cline{2-8}

&10000&1371/457& 4380/177 & 3962/163  & 3523/146 & 2854/120 & \textbf{1918/81} \\ \cline{1-8}

&3000& 555/658 & 242/265 & 264/233  & 217/194 & 166/148 & \textbf{99/90} \\ \cline{2-8}

rank=30  &5000&980/750& 545/227 & 475/204  & 394/174 & 311/137 & \textbf{176/86} \\ \cline{2-8}

&8000 & 1557/606 &1221/204 & 1138/186 & 910/162 &743/130 & \textbf{476/84} \\ \cline{2-8}

&10000&2113/562& 5735/197 & 5040/179  & 4311/157 & 3504/127 & \textbf{2279/83} \\ \cline{1-8}
\end{tabular}
\caption{Results of the runtime and number of iterations with $p=0.08$.}\label{tab-mcp1}
\end{table}

\begin{table}[htb]
\centering
\begin{tabular}{|c|c|cccccc|} 
\hline
\multicolumn{1}{|c|}{rank}&\multicolumn{1}{|c|}{size}
&\multicolumn{6}{|c|}{ Relative error ( $10^{-4} )$ } \\
\hline
&&~SVP~&~DR~&~$\alpha = 1.9$~&~$\alpha = 1.8$~&~$\alpha = 1.7$~&~$\alpha = 1.6$\\ \cline{2-8}
&~3000~& 1.41&1.36 & ~1.32  & ~1.32 & ~1.30 & ~\textbf{1.25} \\ \cline{2-8}

rank=10 &~5000&1.27&1.22 & ~1.17 & ~1.17 & ~1.14 & ~\textbf{1.08}  \\\cline{2-8}

 &~~8000~ &1.18&1.16 & ~1.15  & ~1.13 & ~1.11 & ~\textbf{1.00} \\\cline{2-8}

&~~10000~&1.15& 1.07 & 0.98  & ~0.96 & ~0.95 & ~\textbf{0.93}  \\\cline{1-8}

&~3000~&1.68& 1.83 & ~1.78  & ~1.75 & ~1.64 & ~\textbf{1.50} \\\cline{2-8}

rank=30  &~5000~&1.57& 1.50 & ~1.46  & ~1.46 & ~1.43 & ~\textbf{1.30} \\\cline{2-8}

&~~8000~ &1.39& 1.32 & ~1.34 & ~1.33 & ~1.23 & ~\textbf{1.22} \\\cline{2-8}

&~~10000~&1.32& 1.28 & ~1.28  & ~1.23 & ~1.24 & ~\textbf{1.16} \\ \cline{1-8}
\end{tabular}
\caption{Results of relative error with $p=0.08$}\label{tab-mcp2}
\end{table}

As can be seen from \cref{tab-mcp1} and \cref{tab-mcp2}, for four different values of $\alpha$, the PDR splitting method is always faster than the DR splitting and SVP method. Moreover the solution quality of PDR is better than DR splitting and SVP method. In particular, when $\alpha = 1.6$, we can see the PDR splitting method clearly outperforms the DR splitting method in terms of both the number of iterations and the solution quality.\\

{\bf Real data.}
We now evaluate our algorithms on the Movie-Lens \cite{data} data set, which contains one million ratings for 3900 movies by 6040 users. Table \ref{mcp1} shows the RMSE (root mean square error) obtained by each method with different rank $r$.  For SVP, we take step size $\eta$ as in \cite{JMD}.  For the classical DR  splitting and  PDR  algorithm, we adopt a heuristic method to choose $\gamma$ as before with $k = 100$.  As shown in Table \cref{mcp1}, the PDR algorithm performs always better than both SVP and DR methods. Note that the PDR algorithm with $\alpha=1.9$ is the best.
\begin{table}[htb]
\centering
\begin{tabular}{|c|c|c|c|c|c|c|c|c|}
\hline
~~rank~~& ~~SVP~~ & ~~DR~~ &~~ $\alpha=1.9$ ~& ~~$\alpha=1.8~$ & ~~$\alpha=1.7$ ~& ~~$\alpha=1.6$~ \\\cline{1-7}
5  &1.05 & 0.84 & \textbf{0.80}  & 0.81 & 0.81 & 0.82  \\\cline{1-7}
10 &0.99 & 0.79 & \textbf{0.75} & 0.75 & 0.76 & 0.76 \\\cline{1-7}
15 &0.96 & 0.76  & \textbf{0.72} & 0.72 & 0.73 & 0.74  \\\cline{1-7}
20 &0.93 & 0.72 & \textbf{0.67} & 0.68 & 0.69 & 0.70 \\\cline{1-7}
25 &0.91 & 0.68 & \textbf{0.65} & 0.66 & 0.66 & 0.67 \\\cline{1-7}
30 &0.88 & 0.65 &\textbf{0.62} & 0.63 & 0.63 & 0.64 \\\cline{1-7}
\end{tabular}
\caption{RMSE obtained by each method with varying rank $r$.}\label{mcp1}
\end{table}


\noindent{\bf Method 2:  Take $f(X) = \frac{1}{2}\|\mathcal{P}_{\Omega}(X) - \mathcal{P}_{\Omega}(M)\|^{2}$ and $g(X) = I_{C(r)}(X)$.} Applying the PDR splitting method \eqref{alg} to problem \eqref{ex3-opt}, we get the following algorithm
\begin{equation}\label{ex3-alg}
\begin{cases}
U^{t+1} = \arg\min_{U} \left\{ \frac{1}{2}\|\mathcal{P}_{\Omega}(U) - \mathcal{P}_{\Omega}(M)\|^{2} + \frac{1}{2\gamma}\|U-X^{t}\|^{2} \right\},\\
\\

V^{t+1} = \arg\min_{V} \left\{ I_{C(r)}(V) + \frac{1}{2\gamma}\|\alpha U^{t+1} - X^{t} - V\|^{2} \right\},\\
\\

X^{t+1} = X^{t} + (V^{t+1} - U^{t+1}).
\end{cases}
\end{equation}
We see from the above algorithm that both subproblems  in \eqref{ex3-alg} can be solved easily. For the first subproblem, whose solution can be obtained by using the optimal conditions. On the other hand, it is well known that the second problem has an explicit solution, i.e., the projector from $V$ to the set $C(r)$.  Thus, the solution of \eqref{ex3-alg} may be expressed as
\begin{equation} 
(PDR)~\begin{cases}
 U^{t+1}=
\begin{cases}
 \frac{1}{1+\gamma}\left( X_{i, j}^{t} + \gamma M_{i, j}\right),  ~~~ & (i, j) \in \Omega, \\
 X_{i, j}^{t},  ~~~~ & (i, j)\notin \Omega, \\
\end{cases}
\\
\\

V^{t+1} = P_{C(r)} (\alpha U^{t+1} - V^{t}),\\

\\
X^{t+1} = X^{t} + (V^{t+1} - U^{t+1}).
\end{cases}
\end{equation}
According to the analysis of the convergence in \cref{sec:convergence}, we know the sequence generated by \eqref{ex3-alg} converges to the critical point of $\frac{1}{2}\|\mathcal{P}_{\Omega}(X) - \mathcal{P}_{\Omega}(M)\|^{2} + I_{C(r)}(X) +  \frac{2-\alpha}{2\gamma} \| X \|^2$. In \cite{CCS}, the authors proposed the following singular value thresholding (SVT) method for solving low rank matrix completion problem:
\begin{equation} 
(SVT) ~ \begin{cases}
Y^{t+1}= \Sigma_{j=1}^{r_t}(\sigma_j^{t} - \tau)u_j^{t}v_j^{t}, \\

X^{t+1}_{ij}=~\begin{cases}
 0, ~~~~ & \textmd{if} ~ (i, j) \not\in \Omega,\\
X^{t}_{ij} + \delta(M_{ij} - Y^{t+1}_{ij}), ~~~~ & \textmd{if} ~ (i, j) \in \Omega, 
\end{cases}
\end{cases} 
\end{equation} 
where $U^{t}$, $\Sigma^{t}$, $V^{t}$ are the singular value decomposition of the matrix $Y^{t}$,  and $u_j^{t}, \sigma_j^{t}, v_{j}^{t}$ are corresponding singular vectors and singular value.  They proved that the sequence generated by SVT method convergence to a solution of the following optimization problem with a additional norm term:
\begin{equation}\label{ssvt}
\begin{aligned}
&\min~~~~\tau \| X \|_{*} + \frac{1}{2} \| X \|_{F}^{2},\\
&s.t.~~~~\mathcal{P}_{\Omega}(X) = \mathcal{P}_{\Omega}(M).
\end{aligned}
\end{equation} 
 Next we will  show the numerical experiments of our Algorithm \eqref{ex3-alg} in the cases $\alpha=1.9, 1.8$ and $1.7$, and compare with the results of SVT. \\
 
{\bf Set parameters.} In all of these  experiments, we use \eqref{mcp-stop} as a stopping criterion. For the SVT method, the parameters $\tau=5 n$ and $\delta=1.2 p^{-1}$ are chosen as in \cite{CCS}. As before, we also adopt a heuristic for both algorithms to update $\gamma$: \\

{\it We initialize $\gamma = k * \gamma_0$ and update $\gamma$ as $\max \{ \frac{\gamma}{2}, 0.9999\cdot \gamma_0 \}$ whenever $\gamma > \gamma_0$, and the sequence satisfies either $\| U^t - U^{t-1} \|_{F}/n > 1000/t$ or $\| U \|_{\infty} > 1e10$.}\\

Here we choose $k = 10^6$ for PDR splitting algorithm. Although $k = 10^6$ is selected large here, when the number of iteration $t$ increases, $\gamma$ will eventually be less than $\gamma_0$, which also guarantees the convergence of the algorithm according to \cref{sec:convergence}. We compute the relative error as in \eqref{mcp-error}. Specifically, for the low sampling ratio $p=0.08$ we recover respectively the matrix of rank =10 and 30 in the different sizes $n=3000, 5000, 8000$ and 10000.  Our computational results are displayed in the following. All of these quantities are averaged over five runs. 

\begin{table}[htbp]
\centering
\begin{tabular}{|c|c|cccc|} 
\hline
\multicolumn{1}{|c|}{rank}&\multicolumn{1}{|c|}{size}
&\multicolumn{4}{|c|}{ Average runtime(s) / iterations}\\
\hline
&&~~SVT~~&~~$\alpha = 1.9$~~&~~$\alpha = 1.8$~~&~~$\alpha = 1.7$~~\\ \cline{2-6}
&3000& 41/92 & 66/104  & 40/64 & \textbf{32/52} \\ \cline{2-6}

rank=10 &5000& 127/74 & 170/101 & 104/62 & \textbf{74/44}  \\\cline{2-6}

 &8000 &330/63   & 440/99 & 279/61 & \textbf{200/43} \\\cline{2-6}

&10000&497/60   & 677/98 & 404/60 & \textbf{297/43} \\\cline{1-6}

&3000& 102/167   & 105/122 & 68/79 & \textbf{98/107} \\\cline{2-6}

rank=30  &5000& 276/111   & 269/112 & 168/69 & \textbf{157/68} \\\cline{2-6}

&8000 & 501/86  & 651/105 &397/65 & \textbf{303/48} \\\cline{2-6}

&10000& 912/78   & 964/103 & 608/64 & \textbf{436/46} \\ \cline{1-6}
\end{tabular}
\caption{Results of the runtime and number of iterations with $p=0.08$}\label{tab-mcp3}
\end{table}

\begin{table}[htbp]
\centering
\begin{tabular}{|c|c|cccc|} 
\hline
\multicolumn{1}{|c|}{rank}&\multicolumn{1}{|c|}{size}
&\multicolumn{4}{|c|}{ Relative error ( $10^{-4} )$ } \\
\hline
&&~~SVT~~&~$\alpha = 1.9$~~&~~$\alpha = 1.8$~~&~~$\alpha = 1.7$\\ \cline{2-6}
&~3000~& 1.38   & ~1.20 & ~\textbf{1.11} & ~1.21 \\ \cline{2-6}

rank=10 &~5000&1.20  & ~1.09 & ~1.06 & ~\textbf{1.03} \\\cline{2-6}

 &~~8000~ &1.17   & ~1.03 & ~\textbf{0.97} & ~1.00 \\\cline{2-6}

&~~10000~& 1.05  & ~1.02 & ~1.03 & ~\textbf{0.88} \\\cline{1-6}

&~3000~& 1.85  & ~1.47 & ~\textbf{1.11} & ~1.70 \\\cline{2-6}

rank=30  &~5000~& 1.51  & ~1.25 & ~\textbf{1.10} & ~1.37 \\\cline{2-6}

&~~8000~ & 1.30 & ~1.20 & ~\textbf{1.08} & ~1.14  \\\cline{2-6}

&~~10000~& 1.24 & ~1.18 & ~1.05 & ~\textbf{0.98} \\ \cline{1-6}
\end{tabular}
\caption{Results of relative error with $p=0.08$}\label{tab-mcp4}
\end{table}

Table \cref{tab-mcp3} compares the runtime and the number of iterations required by various methods to reach the stopping criterion  \eqref{mcp-stop} for rank= 10 and 30 in different sizes of matrix when $p=0.08$.  Clearly, Our algorithms are substantially faster than the SVT method. In particular, when $\alpha=1.7$,  our algorithm performs extremely fast in these experiments. Moreover, from the Table \ref{tab-mcp4}, we know that the relative error of PDR algorithm is also lower than the others. Both in terms of time and relative error, we can see from \cref{tab-mcp1}, \cref{tab-mcp2}, \cref{tab-mcp3} and \cref{tab-mcp4} that {\bf Method 2} (with regularization terms) is better than {\bf Method 1} (without regularization terms) and better than the DR algorithm, which also shows the importance of regularization terms. \\

\section{Concluding remarks}
\label{sec:conclude}

In this paper, we apply a parameterized  Douglas-Rachford splitting method for solving several nonconvex optimization problems arising in machine learning.  By constructing  the new parameterized Douglas-Rachford merit function, we establish the global convergence of the parameterized  Douglas-Rachford splitting method when the parameters $\gamma$ and $\alpha$ satisfy \eqref{gammcond} and the sequence generated has a cluster point. We also give sufficient conditions to  guarantee the boundedness of the sequence generated by the proposed method and thus the existence of cluster points.  Finally, we apply our parameterized DR splitting method to three important classes of nonconvex optimization problems: sparsity constrained least squares problem, feasibility problem and low rank matrix completion. The numericial experiments indicate that the parameterized DR splitting method is significantly  better than  the Douglas-Rachford, Peaceman–Rachford splitting methods for the sparsity constrained least squares problem and feasibility problem.  The numerical experiments  also show that the parameterized DR splitting method outperforms some classical methods for the low rank matrix completion in terms of runtime, number of iterations and relative error.

%

\bibliographystyle{siamplain}
\bibliography{reference}


\end{document}